\documentclass[11pt]{amsart}
\usepackage[english]{babel}
\usepackage{amssymb, amsmath, amsthm}
\usepackage[utf8]{inputenc} 
\usepackage[T1]{fontenc}
\usepackage{upgreek, tikz, mathdots}
\usepackage[shortlabels]{enumitem}
\usepackage{hyperref}
\usepackage{amsaddr}
\usepackage[top=1in, bottom=1in, left=1.25in, right=1.25in]{geometry}

\newcommand{\A}{\mathcal{A}}
\newcommand{\PP}{\mathcal{P}}
\newcommand{\R}{\mathcal{R}}
\newcommand{\K}{\mathcal{K}}
\newcommand{\LL}{\mathcal{L}}
\newcommand{\N}{{\mathbb N}}
\newcommand{\Z}{{\mathbb Z}}
\newcommand{\Q}{{\mathbb Q}}
\newcommand{\C}{{\mathbb C}}

\newtheorem{thm}{Theorem}[section]
\newtheorem{prop}[thm]{Proposition}
\newtheorem{lem}[thm]{Lemma}
\newtheorem{cor}[thm]{Corollary}

\theoremstyle{definition}
\newtheorem{ex}{Example}[section]
\newtheorem{rem}{Remark}[section]
\newtheorem{que}{Question}

\numberwithin{equation}{section}

\begin{document}
\title{On $k$-regularity of sequences of valuations and last nonzero digits}

\author{Bartosz Sobolewski}
\address{Institute of Mathematics, Faculty of Mathematics and Computer Science, Jagiellonian University in Krak\'{o}w, \L{}ojasiewicza 6, 30-348, Krak\'{o}w, Poland}
\email{bartosz.sobolewski@uj.edu.pl}

\begin{abstract}
Let $b \geq 2$ be an integer base with prime factors $p_1, \ldots, p_s$. In this paper we study sequences of ``$b$-adic valuations'' and last nonzero digits in $b$-adic expansions of the values $f(n) = (f_1(n), \ldots, f_s(n))$, where each $f_i$ is a $p_i$-adic analytic function. We give a complete classification concerning $k$-regularity of these sequences, which generalizes a result for $b$ prime obtained by Shu and Yao. As an application, we strengthen a theorem by Murru and Sanna on $b$-adic valuations of Lucas sequences of the first kind. Moreover, we derive a method to determine precisely which terms of these sequences can be represented by certain ternary quadratic forms.
\end{abstract}

\maketitle

\section{Introduction} \label{sec:intro}

Let $k \geq 2$ be an integer. We say that a sequence $\mathbf{a} =(a_n)_{n \geq 0}$ is $k$-automatic if its $k$-kernel 
$$ \K_k(\mathbf{a}) = \{ (a_{k^i n +j})_{n \geq 0} : i \geq 0, \: 0 \leq j \leq k^j-1 \}$$
is a finite set. Automatic sequences were initially introduced by Cobham \cite{Cob72} by means of finite automata. Allouche and Shallit \cite{AS92,AS03b} generalized this notion to so-called $k$-regular sequences. When the sequence $\mathbf{a}$ takes values in a $\Z$-module $R$, it is said to be $k$-regular if the $\Z$-submodule of $R^{\N}$ generated by the $k$-kernel of $\mathbf{a}$ is finitely generated (we use the convention $\N = \{0,1,2,\ldots\}$). We say that a sequence is automatic (resp.\ regular) if it is $k$-automatic (resp.\ $k$-regular) for some $k$. 

There has been a lot of interest in studying $p$-regularity of sequences of $p$-adic valuations of subsequent terms of linear recurrence sequences, where $p$ is prime. For the Fibonacci sequence $(F_n)_{n \geq 0}$ Lengyel \cite{Len95} essentially proved that $(\nu_p(F_n))_{n \geq 0}$ is $p$-regular for all primes $p$ (without phrasing it that way). Later Medina and Rowland \cite{MR15} computed the rank of this sequence, that is, the rank of the $\Z$-submodule generated by its $p$-kernel. Similar results for general nondegenerate Lucas sequences of the first kind $(u_n)_{n \geq 0}$, were obtained by Sanna \cite{San16} and Murru,  Sanna \cite{MS18}. Recall that Lucas sequences of the first kind are defined by $u_0=0, u_1 = 1$ and $u_{n+2} = Au_{n+1}+Bu_n$ for $n\in \N$, where $A,B \in \Z$ are fixed.

Also in \cite{MS18} the $b$-adic valuation $\nu_b$ of the terms $u_n$ was studied, where $\nu_b(0) = +\infty$ and for $m \in \Z \setminus \{0\}$ we define
\begin{equation} \label{eq:b_adic_valuation}
\nu_b(m) = \sup \{v \geq 0: \ b^v \mid m\}.
\end{equation}
Although $\nu_b$ is not a valuation in the strict sense unless $b$ is prime, we are going use to this abuse of terminology by analogy with $\nu_p$.
In the context of the present paper, the following result is particularly interesting (with the original notation slightly altered).
\begin{thm}[Murru, Sanna \cite{MS18}] \label{thm:MS18}
If $b \geq 2$ is an integer relatively prime to $B$, then $(\nu_b(u_{n+1}))_{n \geq 0}$ is $b$-regular.
\end{thm}

Regularity of $p$-adic valuations of linear recurrence sequences of higher order has also been investigated.
The $2$-adic valuation of a family of generalized Fibonacci sequences $(T_n(k))_{n \geq 0}$ of order $k \geq 2$ turns out to be $2$-regular when $k=3$ or $k$ is even, which quickly follows from the results of Lengyel and Marques \cite{LM14,LM17} and the author \cite{Sob17}. The more complicated case of $k \geq 5$ odd and the shifted values $T_n(k) + 1$ were studied by Young \cite{You18, You20}. Results by Bravo, D\'{\i}az, and Ram\'{\i}rez \cite{BDR20} on so-called Tripell sequence imply that for $p=2,3$ the corresponding sequences of $p$-adic valuations are $p$-regular.

The results mentioned above (except for Theorem \ref{thm:MS18}) illustrate a general theorem by Shu and Yao \cite{SY11}, stated in the setting of $p$-adic analysis. For a prime number $p$ let $\Z_p$ be the ring of $p$-adic integers, $\Q_p$ its field of fractions, and $\C_p$ the topological completion (with respect to the $p$-adic norm $| \cdot |_p$) of a fixed algebraic closure of $\Q_p$.  We say that a function $f \colon \Z_p \to \C_p$ is locally analytic if it can be expanded into a Taylor series in a neighborhood of any point in $\Z_p$. 

\begin{rem}
In \cite{SY11} such functions $f$ were simply called analytic. In the present paper, however, by an analytic function we mean a function $f$ given by a Taylor series convergent in all $\Z_p$. 
\end{rem}

\begin{thm}[Shu, Yao \cite{SY11}] \label{thm:SY11}
Let $f \colon \Z_p \to \C_p$ be a locally analytic function which does not have any root in $\N$. Then the sequence $(\nu_p(f(n)))_{n \geq 0}$ is $p$-regular if and only if all the roots of $f$ in $\Z_p$ are contained in $\Q$.
\end{thm}

Using interpolation of linear recurrence sequences (along arithmetic progressions) by $p$-adic analytic functions combined with Theorem \ref{thm:SY11}, Shu and Yao supplied some easy to check criteria for $p$-regularity of $(\nu_p(s_n))_{n \geq 0}$ when $(s_n)_{n \geq 0}$ is a sequence satisfying a quadratic linear recurrence. In particular, they managed to prove $p$-regularity of $p$-adic valuations for Lucas sequences of the first and second kind for every prime $p$ (without giving explicit formulas). 

In view of Theorem \ref{thm:MS18} and the above discussion, it seems natural to try to extend Theorem \ref{thm:SY11} to any base $b \geq 2$. One of the primary goals of this paper is to obtain such a generalization. More precisely, let $p_1, \ldots, p_s$ be the prime factors of $b$ and for $i=1,\ldots,s$ let $f_i \colon \Z_{p_i} \to \Q_{p_i}$ be analytic. Assuming that for all $n \in \N$ we have
\begin{equation} \label{eq:common_value}
f_1(n) = \cdots = f_s(n) \in \Z 
\end{equation}
and letting $f(n)$ these common values, the terms $\nu_b(f(n))$ are well-defined and for $k \geq 2$ we can study $k$-regularity of the resulting sequence.
However, in the course of our investigation, we have found that the assumption \eqref{eq:common_value} can be relaxed, yielding more general results with simpler statements. This requires us to give meaning to $\nu_b$ evaluated at tuples of the form $x = (x_1,\ldots,x_s)$, where the $i$-th component is a $p_i$-adic number. A precise definition is given in Section \ref{sec:basic}. Summing up, we want to answer the following question.
\begin{que} \label{que:valuation}
Let $f = (f_1, \ldots, f_s)$, where for each $i=1,\ldots,s$ the function $f_i \colon \Z_{p_i} \to \Q_{p_i}$ is analytic. When is $(\nu_b(f(n)))_{n \geq 0}$ a $k$-regular sequence?
\end{que}

Apart from valuations, we simultaneously consider another family of functions characterizing base-$b$ expansions, which can be collectively called last nonzero digits. More precisely, for a fixed integer $d \geq 1$ and any $n \in \Z$ we define
$$ \LL_b(n) = \begin{cases}
b^{-\nu_b(n)}n &\text{if } n \neq 0, \\
0   &\text{if } n = 0.
\end{cases}$$
and
$$\ell_{b,d}(n) = \LL_b(n) \bmod{b^d},$$
where $x \bmod m$ denotes the remainder of $x \in \Z$ from division by $m$. For example, writing $2400 = 15040_6$ in base $6$, we have $\nu_6(2400) = 1,$ and hence
\begin{align*}
\LL_6(2400) &= 1504_6 = 400, \\
\ell_{6,2}(2400) &= 04_6 = 4.
\end{align*}
It is clear that for $n \in \N$ the value $\LL_b(n)$ is the integer represented by deleting all the trailing zeros from the base-$b$ expansion of $n$. Similarly, $\ell_{b,d}(n)$ is represented by the last block of $d$ digits not ending with $0$.
In short, we will write $\ell_b = \ell_{b,1}$ to denote the single last nonzero digit.

A number authors have investigated the behavior of $\ell_{b,d}$, evaluated at the terms of interesting integer sequences. The most extensively studied example was the sequence $(\ell_{b,d}(n!))_{n \geq 0}$, with most emphasis put on its automaticity and the frequencies with which all possible values occur. One can find relevant results for various values of $b$ and $d$ in the works of Kakutani \cite{Kak67}, Dekking \cite{Dek80}, Dresden \cite{Dre01, Dre08}, Deshouillers and Luca \cite{DL10}, Deshouillers and Ruzsa \cite{DR11}, Deshouillers \cite{Des12, Des16}, Lipka \cite{Lip19}, the author of the present paper \cite{Sob19}, Byszewski and Konieczny \cite{BK20}. In all cases the sequence $(\ell_{b,d}(n!))_{n \geq 0}$ turns out to either be $p$-automatic for some prime factor $p$ of $b$, or coincide with such a sequence on a set of asymptotic density $1$. We would like to highlight the result of Deshouillers and Luca \cite{DL10} who apply a characterization of last nonzero digits of $n!$ to deduce how often it can be expressed as a sum of three squares of integers. This is a natural approach, as by the famous Legendre's three-square theorem this condition is equivalent to $\ell_{4,2}(n!) \not\in \{7,15\}$.

Last nonzero digits of terms other than $n!$ have been studied as well, although to a lesser extent. The classification of Fibonacci and Lucas numbers expressible as a sum of three squares was given by Robbins \cite{Rob83} and rediscovered by Latushkin and Ushakov \cite{LU12}.
Dresden \cite{Dre08} showed transcendence of three real numbers, whose decimal expansions are given by $\ell_{10}(n^n)$, $\ell_{10}(F_n)$, $\ell_{10}(n!)$. Grau and Oller-Marc\'{e}n \cite{GO14} essentially proved that the sequence $(\ell_p(n^n))_{n \geq 0}$ is $p$-automatic for $p$ prime. Some results on periodicity and automaticity of last nonzero digits of sequences of combinatorial origin can be found in the papers by Miska and Ulas \cite[Corollary 4.8]{MU20}, and also Ulas and {\.Z}mija \cite[Theorems 3.5 and 4.1]{UZ21}.

Motivated by the above results, we are interested in the behavior last nonzero digits of linear recurrence sequences in any base $b$. Under the same premise as before, our goal is to answer the following questions.

\begin{que} \label{que:last_nonzero_1}
When is $(\LL_b(f(n)))_{n \geq 0}$ a $k$-regular sequence?
\end{que}

\begin{que} \label{que:last_nonzero_2}
When is $(\ell_{b,d}(f(n)))_{n \geq 0}$ a $k$-automatic sequence? 
\end{que}

We now give an outline of the remainder of the paper. A complete classification, addressing each of the Questions \ref{que:valuation}--\ref{que:last_nonzero_2}, is provided in Section \ref{sec:main}. Our results reveal that the answers to all posed questions are intimately tied to each other. We also give a few simple examples. Section \ref{sec:automatic} contains preliminaries on regular and automatic sequences. In Section \ref{sec:basic} we extend the definitions of $\nu_b, \LL_b,$ and $\ell_{b,d}$ from integers to tuples of $p$-adic numbers and study their general properties. Section \ref{sec:examples} is dedicated to the applications of our methods and contains more complex examples. In particular, we generalize Theorem \ref{thm:MS18} and establish an analogous result for last nonzero digits. Moreover, we provide a general method of determining which terms of a given Lucas sequence of the first kind can be expressed as a sum of three squares. The remaining sections are devoted to the proofs of our main results.

This paper is partially based on a chapter of the author's PhD thesis \cite[Chapter 3]{Sob21}.

\section{Main results} \label{sec:main}

Let $b \geq 2$ be an integer base with prime factorization
$$ b = p_1^{l_1} p_2^{l_2} \cdots p_s^{l_s}, $$
where $p_1, \ldots, p_s$ are distinct primes, and $l_1, \ldots l_s$ positive integers. Also  let $b_i = p_i^{l_i}$ for each $i=1,\ldots,s$.
Consider an $s$-tuple $f = (f_1, \ldots, f_s)$, where for each for $i=1,\ldots,s$ the function $f_i \colon \Z_{p_i} \to \Q_{p_i}$ is analytic. Denoting
\begin{align*}
\Z_b &= \Z_{p_1} \times \cdots \times \Z_{p_s}, \\
\Q_b &= \Q_{p_1} \times \cdots \times \Q_{p_s},
\end{align*}
we can treat $f$ as a function from $\Z_b$ to $\Q_b$. We identify $\Q$ with its embedding in $\Q_b$ via the map $x \mapsto (x, \ldots, x)$, and in the sequel simply write $x$ to denote a rational element of $\Q_b$. In particular, when $n \in \N$ by $f(n)$ we mean $f(n,\ldots,n) = (f_1(n), \ldots, f_s(n))$.

We first identify and exclude some degenerate cases from the further considerations. If all $f_i$ are identically zero, then $\nu_b(f(n)) = \nu_b(0) = + \infty$ and  $\LL_b(f(n)) = \ell_{b,d}(f(n)) = 0$ for any $d \geq 1$ and all $n \in \N$. The following proposition lets us deal with the situation when some (but not all) of the $f_i$ are zero.

\begin{prop} \label{prop:not_zero}
Assume that $f_i = 0$ for some $i \in \{1,\ldots,s\}$ and let $\overline{b} = b/b_i$ and $\overline{f} = (f_1,\ldots,f_{i-1}, f_{i+1}, \ldots, f_s)$. Then for any integer $k \geq 2$ we have the following:
\begin{enumerate}[label={\textup{(\roman*)}}]
\item the sequence $(\nu_b(f(n)))_{n \geq 0}$ is $k$-regular if and only if $(\nu_{\overline{b}}(\overline{f}(n)))_{n \geq 0}$ is $k$-regular;
\item the sequence $(\LL_b(f(n)))_{n \geq 0}$ is $k$-regular if and only if $(\LL_{\overline{b}}(\overline{f}(n)))_{n \geq 0}$ is $k$-regular;
\item the sequence $(\ell_{b,d}(f(n)))_{n \geq 0}$ is $k$-automatic if and only if $(\ell_{\overline{b},d}(\overline{f}(n)))_{n \geq 0}$ is $k$-automatic.
\end{enumerate}
\end{prop}
This motivates us to restrict our attention to the set
$$
\A_b =\{ f = (f_1, \ldots,f_s): \ f_i \colon \Z_{p_i} \to \Q_{p_i} \text{ is analytic and } f_i \neq 0 \text{ for } i=1,\ldots,s\}.
$$

Furthermore, it turns out that the sequence $(\LL_b(f(n)))_{n \geq 0}$ cannot be $k$-regular for any $k \geq 2$ unless all $f_i$ are polynomials.
\begin{prop} \label{prop:regular_polynomial}
Let $f = (f_1,\ldots,f_s) \in \A_b$ be such that the sequence $(\LL_b(f(n)))_{n \geq 0}$ is regular. Then for each $i=1,\ldots,s$ the function $f_i$ is a~polynomial.
\end{prop}
Therefore, we distinguish a subset $\PP_b \subset \A_b$ consisiting of $s$-tuples of polynomials:
$$ \PP_b = \{ f = (f_1, \ldots,f_s) \in \A_b: \ f_i \in \Q_{p_i}[X] \text{ for } i=1,\ldots,s \}. $$

We introduce some further notation and terminology needed to state our results.
For a prime $p$ and an analytic function $f \colon \Z_p \to \Q_p$ we let $\R_f \subset \Z_p$ denote the set of $p$-adic integer roots of $f$. By virtue of  Strassman's Theorem (see \cite[Theorem 4.4.6]{Gou97}) this set is finite. Also for any $\theta \in \R_f$ we define $m_f(\theta)$ to be the multiplicity of $\theta$ as a root of $f$. In general, when $f = (f_1,\ldots,f_s) \in \A_b$ we let $\R_f = \R_{f_1} \times \cdots \times \R_{f_s} \subset \Z_b$.

In the statements of our main results we adhere to a classification which follows from the celebrated theorem of Cobham \cite{Cob69} and its generalization by Bell \cite{Bel06} (both formulated as a part of Proposition \ref{prop:mult_dep} below). More precisely, a sequence $(a_n)_{n \geq 0}$ can be:
\begin{enumerate}[label={(\alph*)}]
\item $k$-regular for all $k \geq 2$;
\item $k$-regular for some $k$ and not $m$-regular for $m$ multiplicatively independent with $k$;
\item not regular.
\end{enumerate}
In particular, this distinction holds for automatic sequences and the result of Cobham says that in the case (a) such a sequence $(a_n)_{n \geq 0}$ must be eventually periodic. In the sequel we say that a sequence is strictly $k$-regular if it falls under case (b). Similarly, we consider strictly $k$-automatic sequences. We point out that this is not standard terminology. 

\begin{rem}
When talking about an eventually periodic sequence $(a_n)_{n \geq 0}$, by its period we mean any positive integer $T$ such that $a_{n+T} = a_n$ for all sufficiently large $n \in \N$. Unless specified otherwise, $T$ is not assumed to be minimal.
\end{rem}

We are now ready to to give the main results of this paper, answering the questions in Section \ref{sec:intro}. The statements are split into two parts, depending on whether $b$ is a prime power or has several prime factors. We start with the sequence $(\nu_b(f(n)))_{n \geq 0}$.

\begin{thm} \label{thm:prime_power_valuation}
Assume that $b = p^l$ is a prime power and let $f \in \A_p$. Then the sequence $(\nu_{p^l}(f(n)))_{n \geq 0}$ is 
\begin{enumerate}[label={\textup{(\alph*)}}]
\item periodic if $\mathcal{R}_f = \varnothing$, where a power of $p$ can be chosen as a period;
\item strictly $p$-regular if $\varnothing \neq \mathcal{R}_f \subset \Q$;
\item not regular if $\mathcal{R}_f \not\subset \Q$.
\end{enumerate}
In particular, the above conditions are independent of $l$.
\end{thm}

For $l=1$ this is a generalization of Theorem \ref{thm:SY11} for functions $f \colon \Z_p \to \Q_p$ in the sense that $k$-regularity  of $(\nu_{p}(f(n)))_{n \geq 0}$ is considered for all $k \geq 2$.

\begin{thm} \label{thm:several_factors_valuation}
Assume that $b$ has $s \geq 2$ prime factors and let $f = (f_1,\ldots,f_s) \in \A_b$. Then the sequence $(\nu_b(f(n)))_{n \geq 0}$ is
\begin{enumerate}[label={\textup{(\alph*)}}]
\item periodic if $\R_{f} = \varnothing$, where a power of $b$ can be chosen as a period;
\item strictly $k$-regular if $\R_{f} = \{\theta\}$ for some $\theta \in \Q \cap \Z_b$, where $k=b_1^{w_1} \cdots b_s^{w_s}$ and $w_1, \ldots, w_s$ are positive integers satisfying 
$$w_1 m_{f_1}(\theta)  = \cdots = w_s m_{f_s}(\theta); $$
\item not regular otherwise.
\end{enumerate}
\end{thm}

The particular choice of $w_1,\ldots,w_s$ satisfying the equality in Theorem \ref{thm:several_factors_valuation} is irrelevant because a $k$-regular sequence is also $k'$-regular when $k,k'$ are powers of the same positive integer (see Proposition \ref{prop:mult_dep}(i) below).

\begin{rem} \label{rem:natural_roots}
In contrast to Theorem \ref{thm:SY11}, here and in the sequel we choose to allow $f$ to have roots in $\N$ even though $\nu_b(0) = +\infty$. One can omit this problem entirely by replacing the (finitely many) zeros in $(f(n))_{n \geq 0}$ with arbitrary values or by shifting this sequence by a sufficiently large integer. Proposition \ref{prop:arithmetic_prog} below guarantees that these operations do not affect regularity of $(\nu_b(f(n)))_{n \geq 0}$.
\end{rem}

We move on to regularity of the sequence $(\LL_b(f(n)))_{n \geq 0}$. In the case $b=p^l$ the conditions on the roots of $f$ determining which of the cases (a)--(c) holds, are exactly the same as in the preceding two results.

\begin{thm} \label{thm:prime_power_last_nonzero}
Assume that $b = p^l$ is a prime power and let $f \in \PP_p$. Then the sequence $(\LL_{p^l}(f(n)))_{n \geq 0}$ is 
\begin{enumerate}[label={\textup{(\alph*)}}]
\item $k$-regular for every $k \geq 2$ if $\mathcal{R}_f = \varnothing$;
\item strictly $p$-regular if $\varnothing \neq \mathcal{R}_f \subset \Q$;
\item not regular if $\mathcal{R}_f \not\subset \Q$.
\end{enumerate}
In particular, the above conditions are independent of $l$.
\end{thm}

When $b$ has several prime factors, in order to have strict regularity of the sequence we additionally need the multiplicities of $m_{f_i}(\theta)$ to be equal.

\begin{thm} \label{thm:several_factors_last_nonzero}
Assume that $b$ has $s \geq 2$ prime factors and let $f = (f_1,\ldots,f_s) \in \mathcal{P}_b$. Then the sequence $(\LL_b(f(n)))_{n \geq 0}$ is
\begin{enumerate}[label={\textup{(\alph*)}}]
\item $k$-regular for every $k \geq 2$ if $\R_{f} = \varnothing$;
\item strictly $b$-regular if $\R_{f} = \{\theta\}$ for some $\theta \in \Q \cap \Z_b$ and  
$$ m_{f_1}(\theta)  = \cdots = m_{f_s}(\theta); $$
\item not regular otherwise.
\end{enumerate}
\end{thm}

The third and final pair of results concerns automaticity of the sequence $(\ell_{b,d}(f(n)))_{n \geq 0}$. The situation turns out to be a bit more complicated when $b = p^l$ is a prime power. Consider a function $f \in \A_p$. For each $\theta \in \R_f$ let $g_{\theta} \in \A_p$ be defined by
$$ f(x) = (x-\theta)^{m_f(\theta)} g_{\theta}(x)  $$
for all $x \in \Z_p$. Let $\lambda$ denote the Carmichael function. Recall that for a positive integer $n$ the value $\lambda(n)$ is the smallest positive integer $m$ such that $a^m \equiv 1 \pmod{n}$ for all integers $a$ coprime with $n$. In particular, for a prime $p$ and exponent $t \geq 1$ we have
$$
\lambda(p^t) = \begin{cases}
p^{t-1}(p-1) &\text{if } p\geq 3, \\
2^{t-2}       &\text{if } p = 2, t \geq 3, \\
2^{t-1}       &\text{if } p = 2, t \leq 2.
\end{cases}
$$
 Consider the subset $\R'_f = \R'_f(l,d) \subset \R_f$, given by
\[	\R'_f(l,d) = \{\theta \in \mathcal{R}_f: l \nmid m_f(\theta) \text{ or } \lambda(p^{ld-\nu_p(g_{\theta}(\theta)) \bmod{l}}) \nmid  m_f(\theta)\}.	\]
In the case $l= 1$ we obtain a simpler form
$$	\mathcal{R}'_f(1,d) = \{\theta \in \mathcal{R}_f:\lambda(p^d) \nmid m_f(\theta) \}.$$
In a sense, the set $\R_f \setminus \R'_f$ contains the roots $\theta$ such that the factor $(n-\theta)^{m_f(\theta)}$ has no effect on automaticity of $(\ell_{p^l,d}(f(n)))_{n \geq 0}$.
It is then no surprise that compared to the previous statements, the set $\R'_f$ essentially replaces $\R_f$.

\begin{thm} \label{thm:prime_power_d_last_nonzero}
Assume that $b = p^l$ is a prime power, $d \geq 1$ an integer, and let $f \in \A_p$. Then the sequence $(\ell_{p^l,d}(f(n)))_{n \geq 0}$ is 
\begin{enumerate}[label={\textup{(\alph*)}}]
\item periodic if $\mathcal{R}'_f = \varnothing$, where a power of $p$ can be chosen as a period;
\item strictly $p$-automatic if $\varnothing \neq \mathcal{R}'_f \subset \Q$;
\item not automatic if $\mathcal{R}'_f \not\subset \Q$.
\end{enumerate}
\end{thm}

When $b$ has several prime factors, the classification again mirrors the one in the earlier theorems.

\begin{thm} \label{thm:several_factors_d_last_nonzero}
Assume that $b$ has $s \geq 2$ prime factors and let $f = (f_1,\ldots,f_s)\in \mathcal{A}_b$. Let $d \geq 1$ be an integer. Then the sequence $(\ell_{b,d}(f(n)))_{n \geq 0}$ is
\begin{enumerate}[label={\textup{(\alph*)}}]
\item periodic if $\R_{f} = \varnothing$, where a power of $b$ can be chosen as a period;
\item strictly $k$-automatic if $\R_{f} = \{\theta\}$ for some $\theta \in \Q \cap \Z_b$, where $k=b_1^{w_1} \cdots b_s^{w_s}$ and $w_1, \ldots, w_s$ are positive integers satisfying 
$$w_1 m_{f_1}(\theta)  = \cdots = w_s m_{f_s}(\theta); $$
\item not automatic otherwise.
\end{enumerate}
In particular, the above conditions are independent of $d$.
\end{thm}

The proofs of all the theorems stated in this section as well as Propositions  \ref{prop:not_zero} and \ref{prop:regular_polynomial} are provided in Sections \ref{sec:prime_power_proofs} and \ref{sec:several_factors_proofs}.

As a corollary, we can relate $b$-regularity of the sequences $(\nu_b(f(n)))_{n \geq 0}$, $(\LL_b(f(n)))_{n \geq 0}$, and $(\ell_{b,d}(f(n)))_{n \geq 0}$ for any base $b \geq 2$.

\begin{cor} \label{cor:prime_power_connection}
Assume that $b = p^l$ is a prime power and let $f \in \A_p$. Then the following conditions are equivalent:
\begin{enumerate}[label={\textup{(\roman*)}}]
\item $\R_f \subset \Q$;
\item the sequence $(\nu_{p^l}(f(n)))_{n \geq 0}$ is $p$-regular;
\item for all $d \geq 1$ the sequence $(\ell_{p^l,d}(f(n)))_{n \geq 0}$ is $p$-automatic.
\end{enumerate}
If additionally $f \in \PP_p$, then the above conditions are also equivalent to the following:
\begin{enumerate}
\item [\textup{(iv)}] the sequence $(\LL_{p^l}(f(n)))_{n \geq 0}$ is $p$-regular.
\end{enumerate}
\end{cor}
\begin{proof}
The condition (i) is equivalent with (ii) and (iv) (when $f \in \PP_p$).
Moreover, as a consequence of Theorem \ref{thm:prime_power_d_last_nonzero}, we have that (iii) follows from (i), and it remains to prove the converse. By the same result, (iii) implies that $\R'_f(l,d) \subset \Q$ for all $d \geq 1$. Observe that for all sufficiently large $d$ the condition $\lambda(p^{ld}) \nmid p^{\nu_p(g_{\theta}(\theta)) \bmod{l}} m_f(\theta)$ in the definition of $\R'_f(l,d)$ is always satisfied, so $\R_f = \R'_f(l,d) \subset \Q$.
\end{proof}

\begin{cor} \label{cor:several_factors_connection}
Assume that $b$ has $s \geq 2$ prime factors and let $f =(f_1,\ldots,f_s) \in \mathcal{A}_b$. Then the following conditions are equivalent:
\begin{enumerate}[label={\textup{(\roman*)}}]
\item either $\R_{f} = \varnothing$ or $\R_{f} = \{\theta\}$ and $m_{f_1}(\theta)  = \cdots = m_{f_s}(\theta) $ for some $\theta \in \Q \cap \Z_b$;
\item the sequence $(\nu_{b}(f(n)))_{n \geq 0}$ is $b$-regular;
\item for all $d \geq 1$ the sequence $(\ell_{b,d}(f(n)))_{n \geq 0}$ is $b$-automatic.
\end{enumerate}
If additionally $f \in \PP_b$, then the above conditions are also equivalent to the following:
\begin{enumerate}
\item [\textup{(iv)}] the sequence $(\LL_{b}(f(n)))_{n \geq 0}$ is $b$-regular.
\end{enumerate}
\end{cor}
\begin{proof}
Equivalence of (i) to all other conditions follows immediately from parts (a), (b) of Theorems \ref{thm:several_factors_valuation}, \ref{thm:several_factors_last_nonzero}, and \ref{thm:several_factors_d_last_nonzero}. 
\end{proof}

We now briefly discuss some implications of our results and provide a few simple examples. First, for two bases $b, b'$ having the same set of prime factors we obtain $\A_b = \A_{b'}$. It is natural to ask whether replacing $b$ with $b'$ in $\nu_b, \LL_b$, and $\ell_{b,d}$ results in a different behavior of the considered sequences, in terms of their regularity. When $b$ has at least two prime factors we see that such a modification does not affect which of the cases (a)--(c) in Theorems \ref{thm:several_factors_valuation}, \ref{thm:several_factors_last_nonzero}, and \ref{thm:several_factors_d_last_nonzero} holds. Only the value of $k$ in the case (b) may change, as demonstrated by the following example.
\begin{ex}
Let $b=50 = 2 \cdot 5^2, b'=20 = 2^2 \cdot 5$ and define $f = (f_1,f_2) \in \PP_{10}$ by $f(x,y) = (x,y^2)$ for $(x,y) \in \Z_2 \times \Z_5 = \Z_{10}$. We see that $\R_f = \{0\}$, so case (b) of Theorems \ref{thm:several_factors_valuation} and \ref{thm:several_factors_d_last_nonzero} holds.
Following the notation of these results, we have $m_{f_1}(0) = 1$ and $m_{f_2}(0)=2$, so taking $w_1=2$ and $w_1=1$, we obtain $m_{f_1}(0) w_1 = m_{f_2}(0) w_2$. It follows that the sequence $(\nu_{50}(f(n)))_{n \geq 0}$ is strictly $100$-regular (equivalently, strictly $10$-regular), while $(\nu_{20}(f(n)))_{n \geq 0}$ is strictly $80$-regular. We obtain similar results concerning automaticity of $(\ell_{50,d}(f(n)))_{n \geq 0}$ and $(\ell_{20,d}(f(n)))_{n \geq 0}$ for any $d \geq 1$. On the other hand, neither $(\LL_{50}(f(n)))_{n \geq 0}$, nor $(\LL_{20}(f(n)))_{n \geq 0}$ are regular, since the multiplicities $m_{f_1}(0) = 1$ and $m_{f_2}(0)=2$ are not equal.
\end{ex}  
When $b = p^l$ is a prime power, replacing $b$ with any other power of $p$ has no effect on regularity in the case of the functions $\nu_b, \LL_b$. However, in Example \ref{ex:base_change} below we show that the case of $\ell_{p^l,d}$ is different, namely modifying $l,d$ may even turn a periodic into a nonautomatic sequence (and vice versa). To this end, recall that for a prime $p \geq 3$, a $p$-adic integer $\sigma$ with $\nu_p(\sigma)=0$ is a square in $\Z_p$ if and only if $\sigma \bmod{p}$ is a square in the finite field $\mathbb{F}_p$. In the case $p=2$ a sufficient and necessary condition is that $\sigma \equiv 1 \pmod{8}$.
This fact follows directly from the famous Hensel's Lemma and a proof can be found in \cite[p. 50]{Rob00}.

\begin{ex} \label{ex:base_change}
Let $p=5$, and $f(x) = 5(x^2 + 1)^4 \in \PP_5$. By the above discussion $-1$ is a~square in $\Z_5$ so $f$ has irrational roots $\theta,-\theta \in \Z_5$ of multiplicity $4$. One can quickly check that $\R'_f(l,d) = \varnothing$ if and only if $(l,d) \in \{(1,1),(2,1)\}$, and otherwise $\R'_f(l,d) = \{\theta, - \theta \} \not\subset \Q$. Therefore, for example the sequence $(\ell_{5^2}(f(n)))_{n \geq 0}$ is periodic but $(\ell_{5,2}(f(n)))_{n \geq 0}$ is not automatic.
\end{ex}

One may also wonder how regularity of $(\nu_b(f(n)))_{n \geq 0}$ relates to regularity of the sequences $(\nu_{b_i}(f_i(n)))_{n \geq 0}$ for $i=1,\ldots,s$. A similar question can also be asked, concering the functions $\LL_b$ and $\ell_{b,d}$. In this regard, the following simple example was at first quite unexpected to the author.

\begin{ex} 
Let $f(x) = x(x+1)$. Then the sequence $(\nu_{p^l}(f(n)))_{n \geq 0}$ is $p$-regular for all prime powers $p^l$. However, $(\nu_{b}(f(n)))_{n \geq 0}$ is not regular for any base $b$ having at least two prime factors. 
\end{ex}  

On the other hand, knowing that $(\nu_{b}(f(n)))_{n \geq 0}$ is strictly $k$-regular for some $k$, we can infer from part (b) of  Theorems \ref{thm:several_factors_valuation} and \ref{thm:prime_power_valuation} that all the sequences $(\nu_{b_i}(f_i(n)))_{n \geq 0}$ are strictly $p_i$-regular. However, if $(\nu_{b}(f(n)))_{n \geq 0}$ is $k$-regular for all $k \geq 2$, then only one of $(\nu_{b_i}(f_i(n)))_{n \geq 0}$ is guaranteed to be $p_i$-regular (strictly or not). In the following example we show that this implication cannot be strengthened.

\begin{ex} \label{ex:relations_b_1}
Let $b = 10$ and consider the polynomial $f(x) = x^2 + 1$. Then $f$ has an irrational root in $\Z_5$ but no root in $\Z_2$. Hence, $(\nu_{10}(f(n)))_{n \geq 0}$  is $k$-regular for every $k \geq 2$. However, out of the two sequences $(\nu_{2}(f(n)))_{n \geq 0}$ and $(\nu_{5}(f(n)))_{n \geq 0}$, only the former is regular. Analogous observations can be made concerning $\LL_b$ and $\ell_{b,d}$.
\end{ex}

Finally, our results allow to easily produce examples of sequences with values in $\Z_b$ whose reductions modulo $b^d$ are automatic, but which are not regular themselves.

\begin{ex}
The sequence $(\LL_{10}(n,n^2))_{n \geq 0} \subset \Z_{10} = \Z_2 \times \Z_5$ is not regular. On the other hand, for each $d \geq 1$ its reduction modulo $10^d$, namely $(\ell_{10,d}(n,n^2))_{n \geq 0}$, is $20$-automatic.
\end{ex}

\section{Preliminaries on automatic and regular sequences} \label{sec:automatic}

In this section we recall some standard facts concerning automatic and regular sequences which will be essential in proving our main results. Whenever regular sequences are concerned, we implicitly assume their terms belong to a $\Z$-module. For a more detailed treatment of the topic we refer the reader to the monograph of Allouche and Shallit \cite{AS03a} and their papers \cite{AS92,AS03b}.

To begin, any $k$-automatic sequence is $k$-regular. Conversely, a $k$-regular sequence which takes on only finitely many values is $k$-automatic. Hence, the results below stated for regular sequences also apply to automatic sequences.

The following proposition justifies the case distinction in the theorems in Section \ref{sec:main}, where the first part of (ii) is the celebrated theorem by Cobham \cite{Cob69}, while (iii) can be extracted from the results of Bell \cite{Bel06}.

\begin{prop} \label{prop:mult_dep}
Let $\mathbf{a} = (a_n)_{n \geq 0}$ be a sequence and $k,l \geq 2$ multiplicatively independent integers. We have the following.
\begin{enumerate}[label={\textup{(\roman*)}}]
\item For any integer $t \geq 1$ the sequence $\mathbf{a}$ is $k$-regular if and only if it is $k^t$-regular.
\item If $\mathbf{a}$ is simultaneously $k$- and $l$-automatic, then it is eventually periodic. Conversely, if $\mathbf{a}$ is eventually periodic, then it is $m$-automatic for all $m \geq 2$.
\item If $\mathbf{a}$ is simultaneously  $k$- and $l$-regular, then it is $m$-regular for all $m \geq 2$.
\end{enumerate}
\end{prop}

We move on to closure properties of automatic and regular sequences. It is known that changing a finite number of terms in a $k$-regular sequence again yields a $k$-regular sequence. We can also relate $k$-regularity of a~sequence and its subsequences along arithmetic progressions.
\begin{prop} \label{prop:arithmetic_prog}
Let $(a_n)_{n \geq 0}$ be a sequence and $b \geq 1$, $k \geq 2$ integers. We have the following.
\begin{enumerate}[label={\textup{(\roman*)}}]
\item If $(a_n)_{n \geq 0}$ is $k$-regular, then for each $c \in \Z$ the subsequence $(a_{bn+c})_{n \geq 0}$ is $k$-regular (arbitrary values can be assigned to terms with $bn+c<0$).
\item If for each $c=0,1,\ldots,b-1$ the subsequence $(a_{bn+c})_{n \geq 0}$ is $k$-regular, then $(a_n)_{n \geq 0}$ is $k$-regular.
\end{enumerate}
\end{prop}

We now look at term-wise operations which preserve regularity.

\begin{prop} \label{prop:function_regular}
Let $(a_n)_{n \geq 0}$ and $(b_n)_{n \geq 0}$ be sequences taking values in a set $R$ and let $k \geq 2$ an integer. Then we have the following.
\begin{enumerate}[label={\textup{(\roman*)}}]
\item If $(a_n)_{n \geq 0}$ is $k$-automatic, then for any set $\Delta$ and function $\rho \colon R \to \Delta$ the sequence $(\rho(a_n))_{n \geq 0}$ is $k$-automatic.
\item If $(a_n)_{n \geq 0}$ is $k$-regular and $\phi \colon R \to S$ is a homomorphism of $\Z$-modules, then the sequence $(\phi(a_n))_{n \geq 0}$ is $k$-regular.
\item If $(a_n)_{n \geq 0}$ and $(b_n)_{n \geq 0}$ are $k$-regular and $R$ is a commutative ring, then for any $\lambda \in R$ the sequences $(\lambda a_n)_{n \geq 0}$, $(a_n + b_n)_{n \geq 0}$, and $(a_n b_n)_{n \geq 0}$ are also $k$-regular.
\item The sequence $((a_n,b_n))_{n \geq 0}$ is $k$-regular if and only if both $(a_n)_{n \geq 0}$ and $(b_n)_{n \geq 0}$ are $k$-regular.
\end{enumerate}
\end{prop}

We have not found a suitable reference for the parts (ii) and (iv) of Proposition \ref{prop:function_regular}, however these assertions quickly follow from the definition of a $k$-regular sequence.
Finally, we state two corollaries which will be used often in our considerations.

\begin{cor} \label{cor:regular_modulo}
Let $(a_n)_{n \geq 0}$ be a~$k$-regular sequence taking values in $\Z_b$. Then for all integers
$m \geq 1$ the sequence $(a_n \bmod b^m)_{n \geq 0}$ is $k$-automatic.
\end{cor}

\begin{cor} \label{cor:polynomial_regular}
Let $R$ be a~commutative ring and $f \in R[X]$. Then the sequence $(f(n))_{n \geq 0}$ is $k$-regular for every $k \geq 2$.
\end{cor}

\section{Basic properties of last nonzero digits} \label{sec:basic}

Let $d \geq 1$ and $b\geq 2$ be integers. As in Section \ref{sec:main}, write the prime factorization of $b$ as
$$ b = p_1^{l_1} \cdots p_s^{l_s}, $$
where $p_1, \ldots, p_s$ are distinct primes, and $l_1, \ldots l_s$ positive integers. Also let $b_i = p_i^{l_i}$ for each $i=1,\ldots,s$. Our first goal is to extend the defininiton of the functions $\nu_b, \LL_b$, and $\ell_{b,d}$ to the product ring 
$$\Q_b = \Q_{p_1} \times \cdots \times \Q_{p_s}.$$
We note that $\Q_b$ only depends on the set of prime factors of $b$, in particular $\Q_{p^l} = \Q_p$ for $p$ prime. Also recall that
$$\Z_b = \Z_{p_1} \times \cdots \times \Z_{p_s}$$
and that we identify $x \in \Q$ with $(x, \ldots, x) \in \Q_b$.

\begin{rem}
The ring $\Z_b$ is usually defined as the inverse limit $ \varprojlim\Z/b^n\Z$ (see  \cite{Ers1,Ers2},  cf.\ \cite[pp. 47--50]{Kat07} for another approach). Equipping each $\Z/b^n\Z$ with the discrete topology, we obtain an isomorphism of topological rings $\Z_b \simeq  \Z_{p_1} \times \cdots \times \Z_{p_s}$, which extends to a natural isomorphism $ \pi_b \colon \Q_b  \to  \Q_{p_1} \times \cdots \times \Q_{p_s}$. The elements of $\Q_b$ can in turn be identified with $b$-adic expansions 
$$
x = \cdots + a_2 b^2 + a_1 b + a_0 + a_{-1} b^{-1} + \cdots + a_{-m} b^{-m}, $$
where $a_j \in \{0,1,\ldots,b-1\}$ and $m \in \N$. In particular, $x \in \Z_b$ if and only if $a_j = 0$ for all $j <0$. The isomorphism $\pi_b$ thus induces the functions $\widetilde{\nu}_b= \nu_b \circ \pi_b$, $\widetilde{\LL}_b = \LL_b \circ \pi_b$, and $\widetilde{\ell}_{b,d} = \ell_{b,d} \circ \pi_b$ defined for $b$-adic expansions. It can be checked that $\widetilde{\LL}_b$ and $\widetilde{\ell}_{b,d}$ (and consequently $\LL_b$ and $\ell_{b,d}$) retain the original interpretation concerning last nonzero digits.
Keeping this in mind, we choose to identify $\Q_b$ with $\Q_{p_1} \times \cdots \times \Q_{p_s}$ for the sake of calculations, as we find it more convenient to deal with tuples of $p$-adic numbers. 
\end{rem}

For any $x = (x_1,\ldots,x_s) \in \Q_b$ let  
$$
\nu_{b}(x) = \min_{1 \leq i \leq s} \nu_{b_i}(x_i) = \min_{1 \leq i \leq s} \left \lfloor \frac{\nu_{p_i}(x_i)}{l_i} \right \rfloor,
$$
where $\lfloor  \cdot\rfloor$ denotes the floor function. It can be easily checked that for any integer $n$ the value $\nu_{b}(n)$ coincides with the original definition \eqref{eq:b_adic_valuation}. By reduction $x \bmod{b^d}$ of $x = (x_1,\ldots,x_s) \in \Z_b$ we mean the unique solution $y \in \{0,1,\ldots,b^d-1\}$ to the system of congruences
$$\begin{cases}
y \equiv x_1 \pmod{b_1^d}, \\
\quad \vdots  \\
y \equiv x_s \pmod{b_s^d}. \\
\end{cases} $$ 
More explicitly, if we put $q_i = b/b_i$, $y_i = x_i$, and let $r_i$ be an integer satisfying $ q_i r_i  \equiv 1 \pmod{b_i^d}$, then for any $x = (x_1, \ldots, x_s) \in \Z_b$ we have 
\begin{equation} \label{eq:CRT}
x \bmod{b^d} = \sum_{i=1}^{s} q_i^d r_i^d  y_i \bmod{b^d}.
\end{equation}
Having extended $\nu_b$ to $\Q_b$, we can now define in general the functions $\LL_b \colon \Q_b \to \Z_b$ and $\ell_{b,d} \colon \Q_b \to \Z/b^d\Z$ precisely as Section \ref{sec:intro}, namely
\begin{align*}
\LL_b(x) &= \begin{cases}
 b^{-\nu_b(x)} x &\text{if } x \neq 0,\\
 0    &\text{if } x = 0,
 \end{cases}\\
\ell_{b,d}(x) &= \LL_b(x) \bmod{b^d}.
\end{align*}

In the following two propositions we give some immediate properties of the $b$-adic valuation and last nonzero digits.
\begin{prop} \label{prop:valuation_properties}
For any $x,y \in \Q_b$ we have the following:
\begin{enumerate}[label={\textup{(\roman*)}}]
\item $\nu_b(xy) \geq \nu_b(x) + \nu_b(y)$;
\item if $b = p^l$ is a prime power, then $\nu_b(xy) \leq  \nu_b(x) + \nu_b(y) + 1$;
\item $\nu_b(x+y) \geq \min\{\nu_b(x), \nu_b(y)\}$ with equality if $\nu_b(x) \neq \nu_b(y)$.
\end{enumerate}
\end{prop}

\begin{prop} \label{prop:multiplicative_properties}
For any $x,y \in \Q_b$ we have the following:
\begin{enumerate}[label={\textup{(\roman*)}}]
\item $\LL_b(bx) = \LL_b(x)$ and  $\ell_{b,d}(bx) = \ell_{b,d}(x)$;
\item if $\nu_b(xy) = \nu_b(x)+\nu_b(y),$ then
\begin{align*}
\LL_b(xy) &= \LL_b(x)\LL_b(y), \\
\ell_{b,d}(xy) &\equiv \ell_{b,d}(x) \ell_{b,d}(y) \pmod{b^d}.
\end{align*}
\end{enumerate}
\end{prop}
Observe that the assumption of Proposition \ref{prop:multiplicative_properties}(ii) is always satisfied for $b$ prime.

In the following proposition we show how to express the function $\ell_{b,d}$ in terms of $\ell_{b_i,d}$. Here and in the sequel we use the convention
\begin{equation} \label{eq:convention}
m^{+\infty} \equiv 0 \pmod{m^d}
\end{equation}
for any integers $d,m \geq 1$.

\begin{prop} \label{prop:last_digits_factorization}
Let $x = (x_1, \ldots,x_s) \in \Q_b \setminus \{0\}$. Then for each $i = 1,\ldots,s$ we have
\begin{equation} \label{eq:last_digits_factorization1}
\ell_{b,d}(x) \equiv b_i^{\nu_{b_i}(x_i)-\nu_b(x)} r_i^{\nu_b(x)} \ell_{b_i,d}(x_i) \pmod{b_i^d}.
\end{equation}
Moreover, for any prime number $p$, integer $l \geq 1$, and nonzero $x \in \Q_p$ we have
\begin{align*}
\LL_{p^l}(x) &= \LL_{p}(x) p^{\nu_{p}(x) \bmod{l}}, \\
\ell_{p^l,d}(x) &\equiv \ell_{p, ld}(x) p^{\nu_{p}(x) \bmod{l}} \pmod{p^{ld}}.
\end{align*}
\end{prop}
\begin{proof}
By definition, for each $i=1,\ldots,s$ we have
$$ \ell_{b,d}\equiv b^{-\nu_b(x)} x_i \pmod{b_i^d}.  $$
The right-hand side can be written as 
$$ b^{-\nu_b(x)} x_i = \begin{cases}  
 b_i^{\nu_{b_i}(x_i)-\nu_{b}(x)} q_i ^{-\nu_b(x)} \LL_{b_i}(x_i) &\text{if } x_i \neq 0, \\
0 &\text{if } x_i = 0.
\end{cases} $$
After reducing modulo $b_i^{d}$ the two cases can be merged into one by the convention \eqref{eq:convention} with $m=b_i$, and thus \eqref{eq:last_digits_factorization1} follows.

In order to prove the second part of the statement, we write $\nu_{p}(x) = l\nu_{p^l}(x) + u$, where $0 \leq u \leq l - 1$. We have 
\[	\LL_{p^l}(x)= \left(p^{-(l \nu_{p^l}(x)  + u)}x \right) p^{u} = \LL_{p}(x) p^{u},\]
as desired. Reduction modulo $p^{ld}$ gives the corresponding formula for $\ell_{p^l,d}(x)$.
\end{proof}

By the equality \eqref{eq:CRT}, we obtain an explicit expression:
\begin{align} 
\ell_{b,d}(x) &\equiv \sum_{i=1}^{s} b_i^{\nu_{b_i}(x_i)-\nu_b(x)}q_i^d r_i^{\nu_b(x)+d} \ell_{b_i,d}(x_i) \pmod{b^d} \label{eq:d_last_nonzero_explicit}\\
&\equiv \sum_{i=1}^{s} p_i^{\nu_{p_i}(x_i)-l_i \nu_b(x)}q_i^d r_i^{\nu_b(x)+d} \ell_{p_i,l_i d}(x_i) \pmod{b^d}.  \nonumber 
\end{align}

The following immediate corollary of Proposition \ref{prop:last_digits_factorization} will be very important later.

\begin{cor} \label{cor:valuation_of_last_digits}
Let $x = (x_1, \ldots,x_s) \in \Q_b$ be nonzero and fix $j \in\{ 1, \ldots, s\}$. Then 
$\nu_{b_j}(\ell_{b}(x)) = 0 \text{ if and only if } \nu_{b_j}(x_j) = \nu_b(x).$
\end{cor}

Combining this with Proposition \ref{prop:function_regular}(i), we deduce that $k$-automaticity of $(\ell_{b}(f(n)))_{n \geq 0}$ for some $k$ implies $k$-automaticity of the characteristic sequence of the set
$$ \{ n \in \N: \ \nu_{b_j}(f(n)) > \nu_b(f(n))  \}.  $$

We conclude this section by showing $b$-regularity of $(\nu_b(n))_{n \geq 0}$ and $(\LL_b(n))_{n \geq 0}$, as well as $b$-automaticity of $(\ell_{b,d}(n))_{n \geq 0}$. This is a special case of the results in Section \ref{sec:main}. As already mentioned in Remark \ref{rem:natural_roots}, we are not concerned with the fact that $\nu_b(0) = + \infty$, as this term can be replaced with any value without affecting regularity.

\begin{prop} \label{prop:last_nonzero_digits_regular}
We have the following:
\begin{enumerate}[label={\textup{(\roman*)}}]
\item the sequence $(\nu_b(n))_{n \geq 0}$ is $b$-regular;
\item the sequence $(\LL_b(n))_{n \geq 0}$ is $b$-regular;
\item the sequence $(\ell_{b,d}(n))_{n \geq 0}$ is $b$-automatic.
\end{enumerate}
\end{prop}
\begin{proof}
For all $n \in \N$ and  $a=1,\ldots,b-1$ we have the relations $\nu_b(bn) =  \nu_b(n)+1$ and $\nu_b(bn+a) = 0$. It follows that the $\Z$-module generated by $\K_b((\nu_b(n))_{n \geq 0})$ is generated by the sequence $(\nu_b(n))_{n \geq 0}$ and the constant sequence $(1)_{n \geq 0}$. This implies (i).

Similarly, for all $n \geq 0$ and  $a=1,\ldots,b-1$ we have $\LL_b(bn) =  \LL_b(n)$ and $\LL_b(bn+a) = bn+a$.  We can thus write
$$ \K_b((\LL_b(n))_{n \geq 0}) = \{(\LL_b(n))_{n \geq 0}\} \cup \bigcup_{a=1}^{b-1} \K_b((bn+a)_{n \geq 0}).$$ For each $a = 1,\ldots,b-1$ the sequence $(bn+a)_{n \geq 0}$ is $b$-regular due to Corollary \ref{cor:polynomial_regular}, and thus the $\Z$-module generated by $\K_b((bn+a)_{n \geq 0})$ is generated by some finite set $S_a \subset \Z^{\N}$. Consequently, the $\Z$-module generated by $\K_b((\LL_b(n))_{n \geq 0})$ is generated by 
$$\{(\LL_b(n))_{n \geq 0}\} \cup \bigcup_{a=1}^{b-1} S_a,$$ 
again a finite set.

Part (iii) follows immediately from (ii) due to Corollary \ref{cor:regular_modulo}.
\end{proof}

\section{Further examples and applications} \label{sec:examples}

In this section we present some more involved examples and give various applications of our main results and methods. To begin, we consider the family of Lucas sequences of the first kind. Recall that $(u_n)_{n \geq 0}$ is a Lucas sequence of the first kind if $u_0 = 0, u_1 = 1$ and 
$$
u_{n+2} = A u_{n+1} + B u_n
$$
for some fixed integers $A,B$ and all $n \in \N$. It is called nondegenerate if the ratio of the complex roots $\alpha, \beta$ of the characteristic polynomial $P(X) = X^2- AX - B$ is not a root of unity. This implies that $\alpha,\beta$ are distinct and the discriminant $\Delta = A^2+4B = (\alpha-\beta)^2$ of $P$ is nonzero.

The following result of Sanna \cite{San16} will be of use.
\begin{thm}[Sanna] \label{thm:San16}
Let $(u_n)_{n \geq 0}$ be a nondegenerate Lucas sequence of the first kind. If $p$ is a~prime number such that $p \nmid B$, then
\[
\nu_p(u_n) = \begin{cases}
\nu_p(n) + \nu_p(u_p) - 1 &\text{if } p \mid \Delta, p \mid n, \\
0 &\text{if } p \mid \Delta, p \nmid n, \\
\nu_p(n) + \nu_p(u_{p\uptau(p)}) - 1 &\text{if } p \nmid \Delta, \uptau(p) \mid n, p \mid n, \\
\nu_p(u_{p\uptau(p)}) &\text{if } p \nmid \Delta, \uptau(p) \mid n, p \nmid n, \\
0 &\text{if } p \nmid \Delta, \uptau(p) \nmid n, 
\end{cases}
\]
for each positive integer $n$, where $\uptau(p) = \min \{n > 0: \ p \mid u_n\}$. 
\end{thm}

For $(u_n)_{n \geq 0}$ nondegenerate we have the Binet-like formula
\begin{equation} \label{eq:Binet}
u_n = \frac{\alpha^n - \beta^n}{\alpha-\beta}.
\end{equation}
This formula can also be considered in the $p$-adic setting, when treating $\alpha,\beta$ as elements of the extension $\Q_p(\sqrt{\Delta}) \subset \C_p$.

We also recall some basic properties of the $p$-adic exponential $\exp_p$ and $p$-adic logarithm $\log_p$ (not to be confused with base-$p$ logarithm). They are defined for $x \in \C_p$ by the usual formulas
\begin{align*}
\exp_p(x) &= \sum_{k=0}^{\infty} \frac{x^k}{k!}, \\
\log_p(1+x) &= \sum_{k=1}^{\infty} (-1)^{k-1} \frac{x^k}{k}.
\end{align*}
Contrary to the classical case $\exp_p(x)$ is convergent for $|x|_p < r_p$, where $r_p = p^{-1/(p-1)}$. The $p$-adic logarithm $\log_p(1+x)$ is convergent for $|x|_p < 1$. Moreover, for $|x|_p < r_p$ and $|y|_p < r_p$ we have the standard identities
\begin{align*}
\exp_p(x+y) &= \exp_p(x) \exp_p(y), \\
\log_p((1+x)(1+y)) &= \log_p(1+x)+\log_p(1+y), \\
\exp_p(\log_p(1+x)) &= 1+x, \\
\log_p(\exp_p(x)) &= x.
\end{align*} 
When $|\gamma-1| < r_p$, we can thus define for $x \in \Z_p$ the function
$$ \gamma^x = \exp_p(x \log_p(\gamma)),  $$ 
which ich analytic on $\Z_p$ and coincides with usual exponentiation for $x \in \Z$. This cannot be immediately applied to the formula \eqref{eq:Binet} as the condition $|\alpha-1|_p=|\beta-1|_p < r_p$ might not be satisfied. However, when $p \nmid B$, there exists a positive integer $\pi_p$ such that $|\alpha^{\pi_p}-1|_p = |\beta^{\pi_p}-1|_p < r_p$ , and thus for each $j = 0,1,\ldots,\pi_p-1$ and $m \in \N$ we can interpolate $ u_{\pi_p m + j} = f_j(m)$, where
$$ f_j(x) = \frac{1}{\alpha-\beta} \left( \alpha^j \exp_p(x \log_p(\alpha^{\pi_p}) - \beta^j \exp_p(x \log_p(\beta^{\pi_p}) \right). $$
When $b \geq 2$ is a base coprime with $B$, we can perform this procedure simultaneously for each prime factor of $b$. More precisely, let $p_1,\ldots, p_s$ be the prime factors of $b$ and let $\pi$ denote the least common multiple of all $\pi_{p_i}$. Then we obtain $\pi$ tuples $f_j = (f_{j,1}, \ldots, f_{j,s}) \in \A_b$ for $j=0,1,\ldots,\pi -1$ such that
$$ u_{\pi m +j} = f_{j,1}(m) = \cdots = f_{j,s}(m).$$
Using this construction and our results, we can  strengthen the theorem of  Murru and Sanna (Theorem \ref{thm:MS18}) concerning $b$-regularity $\nu_b(u_n)$, by adding the word ``strictly''. Moreover, we immediately obtain a similar statement about last nonzero digits of $u_n$.

\begin{cor} 
Let $(u_n)_{n \geq 0}$ be a nondegenerate Lucas sequence of the first kind. If $b$ and $B$ are relatively prime, then:
\begin{enumerate}[label={\textup{(\roman*)}}]
\item the sequence $(\nu_b(u_n))_{n \geq 0}$ is strictly $b$-regular;
\item for every $d \geq 1$ the sequence $(\ell_{b,d}(u_n))_{n \geq 0}$ is strictly $b$-automatic.
\end{enumerate}
\end{cor}
\begin{proof}
By Theorem \ref{thm:MS18} and Proposition \ref{prop:arithmetic_prog}, for each $j = 0,1,\ldots,\pi-1$ the subsequence $(\nu_b(u_{\pi n+j}))_{n \geq 0}$ is $b$-regular. More precisely, Theorem \ref{thm:San16} implies that $\R_{f_0} = \{0\}$ and $\R_{f_j} = \varnothing$ for $j \neq 0$. Hence, by Theorem \ref{thm:prime_power_valuation} or \ref{thm:several_factors_valuation} (depending on the number of prime factors of $b$) the sequence  $(\nu_b(u_{\pi n}))_{n \geq 0}$ is strictly $b$-regular. This implies (i), again by Proposition \ref{prop:arithmetic_prog}.

If $b$ has $s \geq 2$ prime factors, part (ii) follows by Theorem \ref{thm:several_factors_d_last_nonzero}. If $b = p^l$ we use the fact that $m_{f_0}(0) = 1$ to conclude that $\R'_{f_0}(l,d) = \{0\}$, and then apply Theorem \ref{thm:prime_power_d_last_nonzero}.
\end{proof}

Another application of our methods is concerned with the representation
$$u_n = x^2+y^2+z^2,$$ 
where $x,y,z \in \Z$. In Section \ref{sec:intro} we have already mentioned the following result of Robbins \cite{Rob83} and Latushkin, Ushakov \cite{LU12} on the representation of $F_n$ as a sum of three squares.
\begin{thm}[Robbins; Latushkin, Ushakov]
The $n$th Fibonacci number $F_n$ is a~sum of three squares of integers if and only if
$$ n \not\in   \{12l + 10: l \in \N  \}  \cup \{ 4^{k+1}(24l + 21): k,l \in \N\}. $$
\end{thm}

We show that a similar statement holds in general for certain nondegenerate Lucas sequences of the first kind.

\begin{thm} \label{thm:Lucas_quadratic} 
Let $(u_n)_{n \geq 0}$ be a nondegenerate Lucas sequence of the first kind such that $u_n \geq 0$ for all $n \in \N$ and $2 \nmid B$. Let $\pi$ be given by
$$
\pi = \begin{cases}
4  &\text{if } 2  \mid A, \\
6  &\text{if } 2 \nmid A.   
\end{cases}
$$
Then the set of $n \in \N$ such that $u_n$ is not a sum of three squares is a finite union of sets of the form 
\begin{equation} \label{eq:first_form}
\{2^t \pi l + j: l \in \N\},
\end{equation}
and
\begin{equation} \label{eq:second_form}
\{ 2^{t+2k} \pi (8l+c): k,l \in \N  \},
\end{equation}
where $t \in \N$, $j \in \{1,\ldots,2^t \pi-1\}$, and $c \in \{0,1,\ldots,7\}$.
\end{thm}
\begin{proof}
To begin, we show that the subsequences $(u_{\pi m + j})_{m \geq 0}$, with $i=0,1,\ldots,  \pi - 1$, can be interpolated by a $2$-adic analytic functions. By the earlier discussion it is sufficient to prove that the $2$-adic valuation of $\alpha^{\pi}-1$ and $\beta^{\pi}-1$ is greater than $1$. Let 
$$  C = \begin{bmatrix}
0 & 1 \\
B & A
\end{bmatrix}$$
be the companion matrix of the recurrence defining $(u_n)_{n \geq 0}$ and observe that $P$ is its characteristic polynomial. Direct computation shows that $C^{\pi} \equiv I \pmod{4}$, where $I$ denotes the $2 \times 2$ identity matrix. Since $\alpha^{\pi}-1, \beta^{\pi}-1$ are the eigenvalues of $C^{\pi}-I$, it follows that $\nu_2(\alpha^{\pi}-1) = \nu_2(\beta^{\pi}-1) \geq 2$. Therefore, $u_{\pi m + j} = f_j(m)$, where for $x \in \Z_2$ we define
$$  f_j(x) = \frac{1}{\alpha-\beta} \left( \alpha^j \exp_2(x \log_2(\alpha^{\pi}) - \beta^j \exp_2(x \log_2(\beta^{\pi}) \right). $$
From Theorem \ref{thm:San16} we deduce that $0$ is the only $2$-adic integer root of $f_0$ and has multiplicity $1$. On the other hand, the functions $f_1, \ldots, f_{\pi -1}$ have no roots in $\Z_2$.

Now, by Legendre's three-square theorem, a nonnegative integer $r$ is not a sum of three squares if and only if $\ell_{2,3}(r) = 7$ and $\nu_2(r)$ is even, or equivalently, $\ell_{4,2}(r) \in \{7,15\}$.
By part (a) of Theorem \ref{thm:prime_power_d_last_nonzero}, for each $j = 1,\ldots, \pi -1$ the sequence $(\ell_{4,2}(u_{\pi m + j}))_{m \geq 0}$ is periodic with period $2^t$ for some $t \in \N$. It follows that for each $j = 1,\ldots, \pi -1$ the set of $n \equiv j \pmod{\pi}$ such that  $u_n$ is not a sum of three squares, is a finite union of sets of the form \eqref{eq:first_form}.

In the case $i=0$ we can write 
$$ f_0(x) =  x g(x),$$
where $g$ has no root in $\Z_2$. The integer $f_0(m)$ is not a sum of three squares if and only if 
\begin{align}
\ell_{2,3}(m) \ell_{2,3}(g(m)) &\equiv 7 \pmod{8}, \label{eq:three_squares_1} \\
\nu_2(m) + \nu_2(g(m)) &\equiv 0 \pmod{2}. \label{eq:three_squares_2}
\end{align} Let $2^T$ be the minimal common period of the sequences $(\ell_{2,3}(g(m)))_{m \geq 0}$ and $(\nu_2(g(m)))_{m \geq 0}$. We consider the solutions of the congruences depending on $m \bmod{2^T}$. If $m \equiv 0 \pmod{2^T}$, they are of the form $m = 2^{t + 2k}(8l + c)$ with $k,l \in \N$, where $t =  T + (\nu_2(g(0)) \bmod{2})$ and $c  \equiv 7 \ell_{2,3}^{-1}(g(0)) \pmod{8}$. Putting $n = \pi m$, we obtain precisely the set \eqref{eq:second_form}.

If $m  \equiv v \pmod{2^T}$ for some $v \in \{1,\ldots,2^T-1\}$ write $m= 2^{T+2} l + 2^T w + v$, where $l\in \N$ and $w \in \{0,1,2,3\}$. In this case we obtain the equalities $\ell_{2,3}(m) = \ell_{2,3}(2^Tw + v)$ and $\nu_2(m) = \nu_2(v)$, which do not depend on $l$. Plugging them into \eqref{eq:three_squares_1} and \eqref{eq:three_squares_2} and replacing $g(m)$ with $g(v)$, we see that for each fixed $w$ there are either no solutions $m \equiv v \pmod{2^T}$ or they are of the form given above. In the latter case, for $n = \pi m$, we obtain a set of the form \eqref{eq:first_form} with $t = T+2$ and $j = \pi(2^T w + v)$.
\end{proof}

Along the same lines, one can prove results for other ternary quadratic forms $q$ such that the set of integers not represented by $q$ can be expressed in terms of last nonzero digits. We refer the interested reader to \cite{BDTT16}, where appropriate conditions are given for several quadratic forms.

Knowing a concrete formula for the $2$-adic valuation of $u_n$, one can explicitly determine the set considered in Theorem \ref{thm:Lucas_quadratic}.

\begin{ex}
Consider the sequence of Pell numbers $(P_n)_{n \geq 0}$, defined by $P_0=0,P_1=1$, and $P_n = 2P_{n-1} + P_{n-2}$ for $n \geq 2$. We will determine which Pell numbers can be written as a sum of three squares by loosely following the proof of Theorem \ref{thm:Lucas_quadratic}. We can write
$$P_n = \frac{\alpha^n - \beta^n}{\alpha-\beta},    $$
where $\alpha = 1+ \sqrt{2}$ and $\beta = 1 - \sqrt{2}$ (treated as elements of $\Q_2(\sqrt{2})$).
Inspecting for $i=1,2,3$ the subsequences $(P_{4m+i})_{m \geq 0}$ modulo $8$, we deduce that all of their terms are sums of three squares.

Moving on to the subsequence $(P_{4m})_{m \geq 0}$, we have $P_{4m} = f(m)$, where
$$f(x) =  \frac{1}{\alpha-\beta} \left( \exp_2(x \log_2(\alpha^{4}) - \exp_2(x \log_2(\beta^{4}) \right) = \sum_{k=1}^{\infty} c_k x^k,   $$
where 
$$ c_k = \frac{1}{(\alpha - \beta) k!} \left( \log_2^k(\alpha^{4}) - \log_2^k(\beta^{4}) \right).   $$
From Theorem \ref{thm:San16} we get the simple formula $\nu_2(P_n) = \nu_2(n)$, and so $f(x) = x g(x)$, where $g$ has no root in $\Z_2$. We immediately obtain that $\nu_2(g(m)) = \nu_2(P_{4m}) - \nu_2(m) = 2$ is constant for all $m \in \N$. As seen in the proof of Theorem \ref{thm:Lucas_quadratic}, we also need to determine a period of $(\ell_{2,3}(g(m)))_{m \geq 0}$.

To this end, we first show that the coefficients $c_k$ have positive $2$-adic valuation. Since $\alpha - \beta = 2\sqrt{2}$, we obtain $\nu_2(\alpha - \beta) = 3/2$. By Legendre's formula for the $p$-adic valuation of a factorial (see for example \cite[pp.\ 241--242]{Rob00}), we obtain $\nu_2(k!) = k - s_2(k)$, where $s_2$ is the sum of binary digits. Finally, direct calculation gives $\nu_2(\alpha^{4}-1) = \nu_2(\beta^{4}-1) = 5/2$, and thus by \cite[Proposition 1 on p.\ 252]{Rob00} we have $\nu_2(\log_2(\alpha^{4})) = \nu_2(\log_2(\beta^{4})) =5/2$.  Combining all of the above, we obtain
$$  \nu_2(c_k) \geq - \frac{3}{2} - k + s_2(k) + \frac{5}{2} k \geq \frac{3}{2}(k-1)  + s_2(k) \geq 1. $$
We claim that $16$ is a period $(\ell_{2,3}(g(m)))_{m \geq 0}$. For any fixed $m \in \N$ we have the Taylor expansion:
\begin{equation} \label{eq:expansion}
g(m+16) = g(m) + \sum_{i=1}^{\infty} \frac{g^{(i)}(m)}{i!} 16^i,
\end{equation}
where $g^{(i)}$ denotes the $i$-th derivative of $g$. Since $g$ has coefficients with $2$-adic valuation at least $1$, so do the functions $g^{(i)}/i!$. Hence,  the $2$-adic valuation of the sum on the right-hand side of \eqref{eq:expansion} is at least $5$. Dividing both sides by $4$ and reducing modulo $8$, we obtain
$$\ell_{2,3}(g(m+16)) = \ell_{2,3}(g(m)),$$ 
as claimed. Knowing this, we can directly compute that $\ell_{2,3}(g(m))$ is in fact constantly equal to $3$.

Therefore, the term $P_{4m}$ is not a sum of three squares if and only if $\nu_2(m) \equiv 0 \pmod{2}$ and $\ell_{2,3}(m) \equiv 7 \cdot 3^{-1} \equiv 5 \pmod{8}$. Such $m$ can be written in the form $m = 2^{2k}(8l+5)$ with $k,l \in \N$. Returning to $n = 4m$, we conclude that $P_n$ is a sum of three squares if and only if $n$ belong to the set
$$ \N \setminus \{4^{k+1}(8l+5): k,l \in \N\}.   $$
\end{ex}

In the following two examples we showcase the computation of last nonzero digits for a base with several prime factors.

\begin{ex} \label{ex:last_digit_1}
Consider the Fibonacci sequence $(F_n)_{n \geq 0}$. Theorem \ref{thm:several_factors_d_last_nonzero} implies that the sequence $(\ell_{10}(F_n))_{n \geq 0}$ is strictly $10$-automatic. We are going to derive an explicit formula for the last decimal digit of $F_n$ in an elementary way.

It is convenient to compute $\ell_{10}$ evaluated at the subsequences $(F_{30m+j})_{m \geq 0}$ for $j=0,1,\ldots,29$. For any $j \neq 0$  it can be checked, by examining the sequence modulo $100$, that $(\ell_{10}(F_{30m+j}))_{m \geq 0}$ is periodic and $10$ is a common period of all these subsequences. 

In the case $j=0$ we will use the congruence \eqref{eq:d_last_nonzero_explicit}, which for $b=10$ and nonzero rational $x$ becomes
\begin{equation} \label{eq:b=10}
\ell_{10}(x) \equiv 2^{\nu_2(x) - \nu_{10}(x)} \cdot 5 + 5^{\nu_5(x) - \nu_{10}(x)}\cdot 6 \cdot 3^{\nu_{10}(x)} \ell_5(x) \pmod{10}.
\end{equation}
By the results of Lengyel \cite{Len95}, we have $\nu_5(F_n) = \nu_5(n)$ and 
$$
\nu_2(F_n) =
\begin{cases}
0  & \text{if } n \equiv 1,2 \pmod{3}, \\
1  & \text{if } n \equiv 3 \pmod{6}, \\
\nu_2(n) + 2  & \text{if } n \equiv 0 \pmod{6}. 
\end{cases}
$$
It follows from the Binet formula that
$$ 2^{n-1} F_n = \sum_{k=0}^{\lfloor (n-1)/2  \rfloor} \binom{n}{2k+1} 5^k, $$
for all $n \geq 1$. As in \cite[Lemma 1]{Len95}, for $k \neq 0$ we deduce that 
$$ \nu_5 \left(\binom{n}{2k+1} 5^k   \right) > \nu_5(n),$$
which in turn implies 
\begin{equation} \label{eq:5_digits}
\ell_5(F_n) \equiv 3^{n-1} \ell_5(n) \pmod{5}.
\end{equation}

We apply the formula \eqref{eq:b=10} to $x = F_{30m}/40m$ rather than $x=F_{30m}$. This approach is more convenient, as $\nu_2(F_{30m}/40m) = \nu_5(F_{30m}/40m) =0$ so all the exponents in \eqref{eq:b=10} vanish.
Using \eqref{eq:b=10} and \eqref{eq:5_digits}, we obtain
$$
\ell_{10}\left( \frac{F_{30m}}{40m} \right) \equiv 5 + 6 \ell_{5}\left( \frac{30m \cdot 3^{30m-1}}{40m} \right) \\
 \equiv (-1)^{m-1}  \pmod{10}. 
$$
We can now recover the value of $\ell_{10}(F_{30m})$ from the congruence
$$ \ell_{10}(F_{30m}) = \ell_{10} \left( \frac{F_{30m}}{40m} \cdot 40m \right) \equiv \ell_{10}\left( \frac{F_{30m}}{40m} \right) \ell_{10}(4m) \equiv (-1)^{m-1}  \ell_{10}(4m) \pmod{10}.$$
In order to obtain an expression in terms of $n = 30m$, we first compute
$$ \ell_{10}(4m) \equiv \ell_{10} \left(\frac{1}{3}\right) \ell_{10}(120m) \equiv 7 \ell_{10}(4n) \pmod{10}.   $$
Taking all into account, we obtain
$$
\ell_{10}(F_n) =
\begin{cases}
\ell_{10}(F_j) &\text{if } n \equiv j \pmod{300} \\
   &\text{for some } j \not\equiv 0 \pmod{30}, \\
 3\ell_{10}(4n) \bmod{10}  &\text{if } n \equiv 0 \pmod{60}, \\  
 7\ell_{10}(4n) \bmod{10}      &\text{if } n \equiv 30 \pmod{60}.         
\end{cases}
$$
\end{ex}

In the following example we provide a sequence of integers $(a_n)_{n \geq 0}$ such that $(\ell_{b}(a_n))_{n \geq 0}$ is strictly $k$-automatic but $b,k$ are multiplicatively independent.

\begin{ex} \label{ex:last_digit_2}
Consider the last nonzero digit of $a_n = (2n-1)(16^n - 4)$ in the base $b=15$.
The sequence $(a_n)_{n \geq 0}$ can be intepolated by the functions $f_3 \in \A_3$ and $f_5 \in \A_5$, defined by
$$f_p(x) = (2x-1) (\exp_p(x \log_p 16) - 4)$$
for $p=3,5$. By the well-known Lifting the Exponent Lemma, for all $n \in \N$ we obtain
$$\nu_3(16^n-4) = \nu_3(4^{2n-1} -1) = \nu_3(2n-1) + 1.$$
At the same time $\nu_5(16^n-4) = 0$. This means $1/2$ is the only $p$-adic integer root of both $f_3$ and $f_5$ with $m_{f_3}(1/2) = 2$ and $m_{f_5}(1/2) = 1$. Theorem \ref{thm:several_factors_d_last_nonzero} implies that the sequence $(\ell_{15}(a_n))_{n \geq 0}$ is strictly $75$-automatic.

In order to compute its terms in a more explicit way, write
\begin{align*}
f_3(x) &= (2x-1)^2 g_3(x), \\
f_5(x) &= (2x-1) g_5(x),
\end{align*}
where $g_p \in \A_p$ for $p=3,5$. By the above, we have $\nu_3(g_3(x))=1$ and $\nu_5(g_5(x))=0$ so
$$ \ell_{15}(a_n) = \ell_{15}(f_3(n),f_5(n)) \equiv \ell_{15}\left((3(2n-1)^2,(2n-1)\right)   \cdot \left( \frac{g_3(n)}{3}, g_5(n)\right)  \pmod{15}.  $$
We have $g_5(n) = 16^n-4 \equiv 2 \pmod{5}$ and it remains to compute $(g_3(n)/3) \bmod{3}$. Write 
$$g_3(x) = \frac{\exp_p(x \log_p 16) - 4}{2x-1} = \sum_{k=0}^{\infty} c_k x^k.  $$
Multiplying by $2x-1$ and comparing the coefficients, we obtain $c_0 = 3$ and 
$$c_k = 2c_{k-1} - \frac{\log_3^k 16}{k!}$$
for $k \geq 1$. We are going to prove that $\nu_3(c_k) \geq 2$ for all $k \geq 1$. Expanding the logarithm, we see that $\nu_3(c_1) \geq 2$. For $k \geq 2$ Legendre's formula gives 
$$\nu_3\left(\frac{\log_3^k 16}{k!} \right) = k - \frac{k - s_3(k)}{2} \geq 2,$$
and thus our claim holds by induction on $k$. Consequently, $g_3(x) = 3 + 9x h_3(x),$ where $\nu_3(h_3(x)) \geq 0$ for all $x \in \Z_3$.
This implies that for all $n \in \N$ we have
$$ \left(\frac{g_3(n)}{3}, g_5(n) \right) \equiv (1,2) \equiv 7 \pmod{15}.$$ 
Finally, we obtain
$$ 
\ell_{15}(a_n) \equiv  7 \ell_{15} \left(3(2n-1)^2, 2(2n-1)\right) \pmod{15}.
$$
\end{ex}

\section{Proofs for $b = p^l$} \label{sec:prime_power_proofs}

The goal of this section is to prove Theorems \ref{thm:prime_power_valuation}, \ref{thm:prime_power_last_nonzero}, and \ref{thm:prime_power_d_last_nonzero}. Propositions \ref{prop:not_zero} and \ref{prop:regular_polynomial} are proved in the next section for a general base $b$. Here, let $b = p^l$, where $p$ is a prime and $l \geq 1$ an integer. We retain the notation used in Section \ref{sec:main}. In particular, we let $\A_p$ denote the set of nonzero analytic functions $f \colon \Z_p \to \Q_p$ and $\PP_p$ its subset consisting of polynomials. The set of $p$-adic integer roots of $f \in \A_p$ is denoted by $\R_f$, while the multiplicity of $\theta$ is $m_f(\theta)$. For the sake of convenience, for a $p$-adic integer $\theta \in \Z_p$ and $t \in \N$ we introduce the notation 
\begin{align*}
\theta[t] &= \theta \bmod{p^t}, \\
\theta\{t\} &= p^{-t}(\theta - \theta[t]).
\end{align*}
In terms of the $p$-adic expansion of $\theta$, the number $\theta[t]$ is represented by its $t$ initial digits, while $\theta\{t\}$ by the remaining digits. 

We now outline the general structure of our reasoning. We first condsider the simpler case when the function $f \in \A_p$ has at most one distinct root in $\Z_p$. If there is no such root, then the corresponding assertion (covered in case (a) of our theorems) is rather easily shown using uniform continuity of $f$. 
On the other hand, if $f$ has precisely one root $\theta \in \Z_p$, it can be written in the form
$$ f(x) = (x- \theta)^{m_f(\theta)} g(x),   $$
where $g \in \A_p$ has no roots in $\Z_p$. This further boils down to studying terms of the form $c (x - \theta)^{m}$, where $c \in \Q_p$, $m \geq 1$, which is our main focus throughout this section. 

For general $f$ the idea is to find an integer $T \geq 0$ large enough so that each of the functions $f(p^T x + a)$ for $a =0,1, \ldots,p^T-1$ has at most one root in $\Z_p$. Combined with the basic properties of regular sequences (Proposition \ref{prop:arithmetic_prog}), this reduces our investigation to the previous case. 

The whole reasoning is split into a number of auxiliary results, whose proofs are rather elementary, though quite technical in some places. At the very end of this section we combine them to prove the main theorems.
We begin with the simplest case, namely when $f$ has no root in $\Z_p$.

\begin{prop} \label{prop:no_root}
Assume that $f \in \A_p$ has no root in $\Z_p$. Then we have the following:
\begin{enumerate}[label={\textup{(\roman*)}}]
\item the sequence $(\nu_{p^l}(f(n)))_{n \geq 0}$ is periodic;
\item when $f \in \PP_p$, the sequence $(\LL_{p^l}(f(n))_{n \geq 0}$ is $k$-regular for every $k \geq 2$;
\item the sequence $(\ell_{p^l,d}(f(n)))_{n \geq 0}$ is periodic.
\end{enumerate}
Moreover, the period in \textup{(i)} and \textup{(iii)} can be chosen to be a power of $p$.
\end{prop}
\begin{proof}
We first prove (i) and (iii) simultaneously. Since $f$ has no root in $\Z_p$, for some $V \in \N$ we have the inequality $\nu_{p^l}(f(x)) < V$ for all $x \in \Z_p$.
By uniform continuity of $f$ there exists $T \in \N$ such that for all $x,y \in \Z_p$ there holds
$$
\nu_{p^l}(f(x + p^Ty) - f(x)) \geq  V + d.
$$
Equivalently, we can write
$$ f(x+p^Ty) = f(x) + p^{l(V+d)}\sigma$$
for some $\sigma \in \Z_p$ depending on $x,y$.
Putting $x = n \in \N$ and $y = 1$, we obtain both $\nu_{p^l}(f(n+p^T)) = \nu_{p^l}(f(n))$ and $\ell_{p^l,d}(f(n+p^T)) = \ell_{p^l,d}(f(n))$, as desired.

To see why (ii) holds we write
$$\LL_{p^l}(f(n)) = p^{-\nu_{p^l}(f(n))} f(n).$$
The first factor on the right-hand side is periodic with respect to $n$, thus $k$-regular for all $k \geq 2$. The same can be said about the second factor due to Corollary \ref{cor:polynomial_regular}. Since the term-wise product of $k$-regular sequences is $k$-regular, we obtain (ii).
\end{proof}

We move on to the expressions of the form $c (x - \theta)^m$, where $c \in \Q_p$, $\theta \in \Z_p$, and $m$ is a positive integer. To begin, we prove $p$-regularity of the considered sequences when $\theta$ is rational.

\begin{prop} \label{prop:rational_root}
Assume that  $\theta \in \mathbb{Q} \cap \Z_p$. Then we have the following:
\begin{enumerate}[label={\textup{(\roman*)}}]
\item the sequence $(\nu_{p^l}(c(n-\theta)^m))_{n \geq 0}$ is $p$-regular;
\item the sequence $(\LL_{p^l}(c(n-\theta)^m))_{n \geq 0}$ is $p$-regular;
\item for all $d \geq 1$ the sequence $(\ell_{p^l,d}(c(n-\theta)^m))_{n \geq 0}$ is $p$-automatic.
\end{enumerate}
\end{prop}
\begin{proof}
Letting $\theta = q/r$ be written in lowest terms with $r$ positive, we can write
$$ c(n-\theta)^m = \frac{c}{r^m} (rn-q)^m. $$
Due to Proposition \ref{prop:arithmetic_prog}, without loss of generality we can assume $\theta=0$.

Staring with $\nu_{p^l}$, consider the subsequences $(\nu_{p^l}(c(p^l n+a )^m))_{n \geq 0}$, where $a=0,1,\ldots,p^l-1$. 
For $a=0$ we have
$$ \nu_{p^l}(c(p^l n)^m) = m + \nu_{p^l}(c n^m).$$
Now take $0 < a <p^l$ and write $a = p^k e$, where $k, e \in \N$ and $\nu_p(e)=0$. Since $k< l$, for all $n \in \N$ we obtain
$$  \nu_{p^l}(c(p^l n+a)^m) = \nu_{p^l}(cp^{km}(p^{l-k}n+e)^m) = \left\lfloor  \frac{\nu_p(c)+km}{l} \right\rfloor,  $$
a constant. Therefore, the $\Z$-module generated by the $k$-kernel of $(\nu_{p^l}(cn^m))_{n \geq 0}$ is generated by $(\nu_{p^l}(cn^m)_{n \geq 0})$ itself and the constant sequence $(1)_{n \geq 0}$.

Moving on to (ii), by Proposition \ref{prop:last_digits_factorization}(ii) and the multiplicative property of $\LL_p$, for all $n \geq 1$ we have
\begin{equation} \label{eq:rational_root}
\LL_{p^l}(c n^m) = p^{(\nu_p(c) + m \nu_p(n)) \bmod{l}} (\LL_p(n))^m .
\end{equation}
The sequence $((\nu_p(c)+m\nu_p(n)) \bmod{l})_{n \geq 0}$ is $p$-automatic as a reduction modulo $l$ of a $p$-regular sequence. We deduce that the first factor on the right-hand side of \eqref{eq:rational_root} is $p$-regular.
By Proposition \ref{prop:last_nonzero_digits_regular}(ii) so is $(\LL_p(n))_{n \geq 0}$, which implies that $(\LL_{p^l}(cn^m))_{n \geq 0}$ is also $p$-regular as a termwise product of $p$-regular sequences. 

Part (iii) follows immediately from (ii) by Corollary \ref{cor:regular_modulo}.
\end{proof}

We ultimately need to prove that the properties in Proposition \ref{prop:rational_root} hold in the strict sense (under extra assumptions in the case of $\ell_{p^l,d}$). The following technical lemma is a step in this direction and will also be useful later in proving nonautomaticity.

\begin{lem} \label{lem:different_val_digits}
Let $\theta, \sigma \in \Z_p$ be such that $\theta \neq \sigma$. Then we have the following:
\begin{enumerate}[label={\textup{(\roman*)}}]
\item for any $t \in \Z \setminus \{0\}$ there exist infinitely many $n \in \N$ such that $\nu_p(n - \theta) = \nu_p(n - \sigma) + t$;
\item for any integer $\delta \geq 1$ and $x \in \{1,\ldots,p^{\delta}-1\}$ such that $p \nmid x$ and $x \neq \ell_{p,\delta}(\theta-\sigma)$ there exist infinitely many $n \in \N$ such that $\ell_{p,\delta}(n-\theta)= x$ and $\ell_{p,\delta}(n-\sigma) = \ell_{p,\delta}(\theta-\sigma)$.
\end{enumerate}
\end{lem}
\begin{proof}
Put $v = \nu_p(\theta - \sigma)$. In (i), by renaming $\theta$ and $\sigma$ we can assume without loss of generality that $t > 0$. It is enough to take any $n$ such that $\nu_p(n- \theta) = v+t$, for example
$$n = p^{v+t+1}j + p^{v+t} + \theta[v+t+1],$$
where $j \geq 1$ is arbitrary. Then $\nu_p(n- \theta) > \nu_p(\theta- \sigma)$ so
$$ \nu_p(n - \sigma) = \nu_p((n- \theta) + (\theta- \sigma)) = v,$$
and (i) follows.

Moving on to (ii), if $p=2$ and $\delta = 1$ there is nothing to prove.
Otherwise, let $$n = p^{v+\delta}(x+p^{\delta}j) + \theta[v+2\delta],$$
where $j \geq 1$ is again arbitrary. We obtain that
$$\frac{n-\theta}{p^{v+\delta}} = x + p^{\delta}(j-\theta\{v+2\delta\}),  $$
which yields $\nu_p(n-\theta) = v+\delta$ and $\ell_{p,\delta}(n-\theta) = x$. Again, we can write $n-\sigma = (n -\theta) + (\theta - \sigma)$ and observe that $\nu_p(n-\theta) = \nu_p(\theta - \sigma)+ \delta$, which gives $\ell_{p,\delta}(n-\sigma)= \ell_{p,\delta}(\theta-\sigma)$.
\end{proof}

In order to prove that the sequence $(\nu_{p^l}(c(n-\theta)^m))_{n \geq 0}$ is strictly $p$-regular, it suffices to show that its reduction modulo some positive integer is not eventually periodic. At the same time, if $(\ell_{p^l,d}(c(n-\theta)^m))_{n \geq 0}$ is eventually periodic, then so is $(\nu_{p}((n-\theta)^m) \bmod{l})_{n \geq 0}$, and we would like to know when the latter possibility can be ruled out. In both situations we will use the following auxiliary result.

\begin{lem} \label{lem:eventually_periodic_valuation}
Let $t,u$ be positive integers such that $u \nmid m$. Then $(\nu_{p^t}(c (n-\theta)^m) \bmod{u})_{n \geq 0}$ is not eventually periodic.
\end{lem}
\begin{proof}
Suppose for the sake of contradiction that $(\nu_{p^t}((n-\theta)^m) \bmod{u})_{n \geq 0}$ is eventually periodic with period $T > 0$. We have
\begin{equation} \label{eq:eventually_periodic_valuation}
\nu_{p^t}(c (n-\theta)^m) = \left \lfloor \frac{\nu_p(c) + m \nu_p(n - \theta)}{t} \right \rfloor.
\end{equation}
Lemma \ref{lem:different_val_digits}(i) applied to $\sigma = \theta - T$ provides infinitely many $n \in \N$ such that $\nu_p(n+T-\theta) = \nu_p(n-\theta) + t$.  Using the formula \eqref{eq:eventually_periodic_valuation} twice, for such $n$ we obtain
$$ \nu_{p^t}(c (n+T-\theta)^m) = \nu_{p^t}(c (n-\theta)^m) + m   $$
so these two valuations are not congruent modulo $u$, thus a contradiction.
\end{proof}

We are now ready to prove a necessary and sufficient condition for eventual periodicity of $(\ell_{p^l,d}(c(n-\theta)^m))_{n \geq 0}$. For later reference we highlight it below:
\begin{equation} \label{eq:condition}
l \mid m \text{ and } \lambda(p^{ld-\nu_p(c) \bmod{l}}) \mid m. \tag{C}
\end{equation}
Note that $\R'_f$, as defined in Section \ref{sec:main}, contains precisely these $\theta$ for which the condition \eqref{eq:condition} with $c = g_{\theta}(\theta), m = m_f(\theta)$ does not hold.

\begin{prop} \label{prop:eventually_periodic_digits}
The condition \eqref{eq:condition} is equivalent to the following:
\begin{enumerate}[label={\textup{(\roman*)}}]
\item the sequence $(\ell_{p^l,d}(c(n-\theta)^m))_{n \geq 0}$ is eventually periodic;
\item $\ell_{p^l,d}(c(n-\theta)^m) = \ell_{p^l,d}(c)$ for all $n \in \N$.
\end{enumerate}
\end{prop}
\begin{proof}
When multiplying $c$ by an integer power of $p^l$ the value $\ell_{p^l,d}(c(n-\theta)^m)$ does not change, hence without loss of generality we can assume that $\nu_p(c) \in \{0,1,\ldots,l-1\}$. 

We first prove the implication \eqref{eq:condition}$\implies$(ii). 
Since $l \mid m$, by Propositions \ref{prop:multiplicative_properties}(ii) and \ref{prop:last_digits_factorization} we have
\begin{equation} \label{eq:eventually_periodic_digits}
\ell_{p^l,d} (c(n - \theta)^m) \equiv \ell_{p^l,d}(c) \ell_{p^l,d}((n - \theta)^m) \equiv  c (\ell_{p,ld}(n - \theta))^m  \pmod{p^{ld}}. 
\end{equation} 
At the same time, $\lambda(p^{ld-\nu_p(c)}) \mid m$ implies
$$ (\ell_{p,ld}(n - \theta))^m \equiv 1 \pmod{p^{ld-\nu_p(c)}}.$$
Combining the two congruences we obtain (ii).

Since the implication (ii)$\implies$(i) is obvious, we are left with proving (i)$\implies$\eqref{eq:condition}. Suppose that the sequence $(\ell_{p^l,d}(c(n-\theta)^m))_{n \geq 0}$ is eventually periodic with period $T$. As $\nu_p(\ell_{p^l,d}(x))=\nu_p(x) \bmod{l}$ for any $x \in \Z_p$ , the sequence $(\nu_p((n-\theta)^m) \bmod{l})_{n \geq 0}$ is eventually periodic as well with the same period. Lemma \ref{lem:eventually_periodic_valuation} applied to $t=1,u=l$ implies that $l \mid m$, and thus again we get 	\eqref{eq:eventually_periodic_digits}.

For the sake of contradiction suppose further that $\lambda(p^{ld-\nu_p(c)}) \nmid m$. This means that there exists $x \in \{1,\ldots,p^{ld-\nu_p(c)}-1\}$ not divisible by $p$ and such that
$x^m \not \equiv (\ell_{p,ld}(T))^m \pmod{p^{ld-\nu_p(c)}}$. Applying Lemma \ref{lem:different_val_digits}(ii) to $\sigma = \theta - T$ and $\delta =ld$ gives infinitely many values of $n$ such that
$$(\ell_{p,ld}(n+T-\theta))^m \equiv (\ell_{p,ld}(T))^m  \not\equiv x^m \equiv  (\ell_{p,ld}(n-\theta))^m \pmod{p^{ld-\nu_p(c)}}.$$
This non-congruence still holds when both sides are multiplied by $c$ and the modulus is changed to $p^{ld}$. Using \eqref{eq:eventually_periodic_digits}, we  obtain $\ell_{p^l,d}(c(n+T-\theta)^m) \neq \ell_{p^l,d}(c(n-\theta)^m)$ for infinitely many $n$, which contradicts periodicity.
\end{proof}

We now turn to the case when $\theta$ is an irrational $p$-adic integer, with the aim to prove nonregularity of the considered sequences. In this regard, the following result is an analogue of Lemma \ref{lem:eventually_periodic_valuation}.

\begin{lem} \label{lem:not_automatic_valuation}
Assume that $\theta\in \Z_p \setminus \Q$  and let $t,u$ be positive integers such that $u \nmid m$. Then the sequence $(\nu_{p^t}(c (n-\theta)^m) \bmod{u})_{n \geq 0}$ is not automatic.
\end{lem}
\begin{proof}
Fix an integer $k \geq 2$ and write $k = p^e r$, where $e, r \in \N$ and $\nu_p(r)=0$. 
By Proposition \ref{prop:mult_dep}(i) without loss of generality we can replace $k$ with $k^{ut}$, and thus assume that $e$ is divisible by $ut$. We will prove directly that the $k$-kernel of the considered sequence is infinite. More precisely, we claim that for $j \in \N$ the subsequences 
$(\nu_{p^t}(c(k^{j}n + \theta[ej] - \theta)) \bmod{u})_{n \geq 0}$
are all distinct.  First, we compute
$$ \nu_p(k^{j}n + \theta[ej] - \theta)  =  ej +  \nu_p(n - r^{-j}\theta\{ej\}),$$
which implies that
\begin{align} \label{eq:not_automatic_valuation}
\nu_{p^t}(c(k^{j}n + \theta[ej] - \theta)^m) &= \left \lfloor \frac{\nu_p(c) + m \nu_p(k^{j}n + \theta[ej] - \theta)}{t} \right \rfloor \nonumber \\
&\equiv \left \lfloor \frac{\nu_p(c) + m \nu_p(n - r^{-j}\theta\{ej\})}{t} \right \rfloor \pmod{u},
\end{align}
where we used $ut \mid e$.

Now, fix any $i,j \in \N$ such that $i \neq j$. Since $\theta$ is irrational, we have $r^{-i}\theta\{ei\} \neq r^{-j}\theta\{ej\}$, and thus by Lemma \ref{lem:different_val_digits}(i) we obtain $n \in \N$ such that 
$$\nu_p(n - r^{-j}\theta\{ej\}) = \nu_p(n - r^{-i}\theta\{ei\}) + t.$$
Plugging this in \eqref{eq:not_automatic_valuation}, we obtain
$$
\nu_{p^t}(c(k^{j}n + \theta[ej] - \theta)^m)  \equiv \nu_{p^t}(c(k^{i}n + \theta[ei] - \theta)^m) + m \pmod{u}.
$$
As $m \nmid u$, we have found an index $n$ at which the two subsequences corresponding to $i,j$ differ, and the result follows.
\end{proof}

It turns out that when $\theta \not\in \Q$, the sequence $(\ell_{p^l,d}(p^v(n-\theta)^m))_{n \geq 0}$ is not automatic at all that unless the condition \eqref{eq:condition} is satisfied.

\begin{prop} \label{prop:not_automatic_digits}
Assume that  $\theta \in \Z_p \setminus \Q$. Then we have the following:
\begin{enumerate}[label={\textup{(\roman*)}}]
\item the sequence $(\nu_{p^l}(c(n-\theta)^m))_{n \geq 0}$ is not regular;
\item the sequence $(\LL_{p^l}(c(n-\theta)^m)))_{n \geq 0}$ is not regular;
\item if \eqref{eq:condition} does not hold, the sequence $(\ell_{p^l,d}(c(n-\theta)^m)))_{n \geq 0}$ is not automatic.
\end{enumerate}
\end{prop}
\begin{proof}
If the sequence $(\nu_{p^l}(c(n-\theta)^m))_{n \geq 0}$ were $k$-regular for some $k$, then its reduction modulo any positive integer $u$ would be $k$-automatic. Lemma \ref{lem:not_automatic_valuation} applied to $t=l$ shows that this is not the case. 

Part (ii) follows immediately from (iii), again by the fact that the reduction modulo $p^{ld}$ of a $k$-regular sequence of $p$-adic integers yields a $k$-automatic sequence.

Therefore, we have left to prove (iii). Similarly as in Proposition \ref{prop:eventually_periodic_digits}, without loss of generality $\nu_p(c) \in \{0,1,\ldots,l-1\}$. If $l \nmid m$, then Lemma \ref{lem:not_automatic_valuation} applied to $t=1, u=l$ shows that $(\nu_p(c(n-\theta)^m) \bmod{l} )_{n \geq 0}$ is not automatic, and the assertion follows. 

Otherwise, if $l \mid m$, we must have $\lambda(p^{ld-\nu_p(c)}) \nmid m$. Fix an integer $k \geq 2$ and write $k = p^e r$, where $e, r \in \N$ and $\nu_p(r)=0$. Raising $k$ to a suitable power, we can assume $r \equiv 1 \pmod{p^{ld}}$. We claim that for $j \in \N$ the subsequences $(\ell_{p^l,d}(c(k^{j}n + \theta[ej] - \theta)^m))_{n \geq 0}$ are all distinct, which further implies that the considered sequence in not $k$-automatic. As $l \mid m$, we the congruence \eqref{eq:eventually_periodic_digits} holds, yielding after some simplifications
\begin{equation} \label{eq:not_automatic_digits}
\ell_{p^l,d}(c(k^{j}n + \theta[ej] - \theta)^m)
 \equiv c (\ell_{p,ld}(n - r^{-j}\theta\{ej\}))^m \pmod{p^{ld}}. 
\end{equation}
Again, we have $r^{-i}\theta\{ei\} \neq r^{-j}\theta\{ej\}$ for any distinct $i,j \in \N$. Hence, by Lemma \ref{lem:different_val_digits} we can find $n \in \N$ such that
$$ (\ell_{p,ld}(n - r^{-i}\theta\{ei\}))^m \not \equiv (\ell_{p,ld}(n - r^{-j}\theta\{ej\}))^m \pmod{p^{ld-\nu_p(c)}}. $$
Multiplying both sides $c$, the modulus by $p^{\nu_p(c)}$ and using \eqref{eq:not_automatic_digits}, we get
$$\ell_{p^l,d}(c(k^{i}n + \theta[ei] - \theta)^m) \not \equiv \ell_{p^l,d}(c(k^{j}n + \theta[ej] - \theta)^m) \pmod{p^{ld}}. $$
The result follows.
\end{proof}

We now state the final auxiliary result, which allows us to reduce studying general  functions $f \in \A_p$ to those having at most one root in $\Z_p$.

\begin{prop} \label{prop:constant_val_digits}
Let $f \in \A_p$ and $\delta \geq 1$ be an integer. 
Then there exists an integer $T \geq 0$ such that for each $a =0,1,\ldots,p^T-1$ the function $f_{a} \in \mathcal{A}_p$, defined by 
$$f_a(x) = f(p^Tx + a),$$
has at most one root in $\Z_p$. More precisely: 
\begin{enumerate}[label={\textup{(\roman*)}}]
\item if $a \neq \theta[T]$ for all $\theta \in \mathcal{R}_f$, then $f_{a}$ has no root in $\Z_{p}$;
\item if $a = \theta[T]$ for some $\theta \in \mathcal{R}_f$, then $f_a$ has precisely one root in $\theta\{T\} \in \Z_p$. 
\end{enumerate}
Furthermore, for each $\theta \in \R_f$ let
$$ f(x) = (x-\theta)^{m_f(\theta)} g_{\theta}(x),  $$
where $g_{\theta} \in \A_p$. Then $T$ can be chosen in such a way that for all $\theta \in \R_f$ and $x \in \Z_{p}$ we have
\begin{align} 
\nu_p(g_{\theta}(p^Tx +\theta[T])) &= \nu_p (g_{\theta}(\theta)), \label{eq:constant_val} \\
\ell_{p,\delta}(g_{\theta}(p^Tx +\theta[T])) &= \ell_{p,\delta}(g_{\theta}(\theta)). \label{eq:constant_digits}
\end{align}
\end{prop}
\begin{proof}
Consider the partition of $\Z_p$ into $p^T$ balls  $a + p^T \Z_p$, where $a =0,1,\ldots,p^T-1$. Let $T_0$ be the minimal value of $T$ such that distinct roots of $f$ lie in distinct balls. Then for any $T \geq T_0$ it is easy to verify that (i) and (ii) hold.

We now show how to refine this choice. For each $\theta \in \R_f$ the function $h_{\theta} \in \A_p$ defined by $h_{\theta}(x) = g_{\theta}(p^{T_0} x + \theta[T_0])$, has no root in $\Z_p$. By Proposition \ref{prop:no_root} the sequences $(\nu_{p}(h_{\theta}(n)))_{n \geq 0}$ and $(\ell_{p,\delta}(h_{\theta}(n)))_{n \geq 0}$ are periodic. Let $T_1$ be such that $2^{T_1}$ is a common period of these sequences for all $\theta \in \R_f$. We claim that for $T = T_1 + T_0$ (or larger) both equalities \eqref{eq:constant_val} and \eqref{eq:constant_digits} are satisfied. For each $\theta \in \R_f$ the functions $\nu_p \circ g_{\theta}$ and $\ell_{p,d} \circ g_{\theta}$ are constant on balls of the form $\theta[T_0] + p^{T_0}(a+ p^{T_1}\Z_p) = \theta[T_0] + p^{T_0} a+ p^T \Z_p$ for all $a = 0,1, \ldots, p^{T_1}-1$. In particular, for $a = p^{-T_0}(\theta[T]-\theta[T_0])$ we obtain the ball $\theta[T] + p^T \Z_p$. Since this ball contains $\theta$, we get \eqref{eq:constant_val} and \eqref{eq:constant_digits}.
\end{proof}

We are now ready to combine all the intermediate results in order to prove Theorems \ref{thm:prime_power_valuation}, \ref{thm:prime_power_last_nonzero}, and \ref{thm:prime_power_d_last_nonzero}. First, the functions $f \in \A_p$ having no root in $\Z_p$ fall under part (a) of each theorem, and Proposition \ref{prop:no_root} implies the corresponding assertions. Therefore, we can focus on $f$ having at least one root in $\Z_p$, in which case we will make use of the following setup in the proofs. Let $T$ be an integer satisfying the assertion of Proposition \ref{prop:constant_val_digits} with $\delta = ld$, and such that $l \mid T$. For each $a =0,1,\ldots,p^T-1$ let $f_a \in \A_p$ be defined by 
$$f_a(x) = f(p^T x +a),$$ 
precisely as in the proposition. Note that the sequence $(\nu_{p^l}(f(n)))_{n \geq 0}$ is periodic with a period being a power of $p$ if and only if all the sequences $(\nu_{p^l}(f_a(n)))_{n \geq 0}$ have the same property. By Proposition \ref{prop:arithmetic_prog}, for fixed $k \geq 2$ a similar observation can be made regarding $k$-regularity of $(\LL_{p^l}(f(n)))_{n \geq 0}$ and $k$-automaticity of $(\ell_{p^l,d}(f(n)))_{n \geq 0}$.

Now, if $a$ is not equal to $\theta[T]$ for any $\theta \in \R_f$, then $f_a$ has no root in $\Z_p$. Therefore, by Proposition \ref{prop:no_root} and the above discussion it is enough to consider the functions $f_{\theta[T]}$, where $\theta$ ranges over $\R_f$. Since $T$ is fixed, we are going to write $f_{\theta}$ as a shorthand for $f_{\theta[T]}$.
With $g_{\theta}$ as in Proposition \ref{prop:constant_val_digits}, we have
\begin{equation} \label{eq:f_a}
f_{\theta}(x) = p^{Tm_f(\theta)} \left(x-\theta\{T\}\right)^{m_f(\theta)} g_{\theta}(p^Tx +\theta[T]).
\end{equation}
If for each $\theta \in \R_f$ we put
$$c_{\theta} = p^{Tm_f(\theta)}g_{\theta}(\theta),$$
then by the equalities \eqref{eq:constant_val}, \eqref{eq:constant_digits}, and \eqref{eq:f_a} together with the assumption $l \mid T$, we obtain for all $n \in \N$ the following:
\begin{align}
\nu_{p^l}(f_{\theta}(n)) &=  \nu_{p^l}(c_{\theta}(n-\theta\{T\})^{m_f(\theta)}), \label{eq:valuation_representation} \\
\LL_{p^l}(f_{\theta}(n)) &= \LL_{p^l}(c_{\theta}(n-\theta\{T\})^{m_f(\theta)}) \cdot   \frac{g_{\theta}(p^Tn +\theta[T])}{g_{\theta}(\theta)}, \label{eq:last_nonzero_representation} \\
\ell_{p^l,d}(f_{\theta}(n)) &= \ell_{p^l,d}(c_{\theta}(n-\theta\{T\})^{m_f(\theta)}). \label{eq:d_last_nonzero_representation}
\end{align}
Having this representation, it is rather easy to finish the proofs by utilizing the earlier results, where $c=c_{\theta}$, $m= m_f(\theta)$, and $\theta$ is replaced with $\theta\{T\}$.

\begin{proof}[Proof of Theorem \ref{thm:prime_power_valuation}]

In the case (b) we have $\R_f \subset \Q$ so for all $\theta \in \R_f$ we also get $\theta\{T\} \in \Q$. Equality \eqref{eq:valuation_representation} together with Proposition \ref{prop:rational_root}(i) implies $p$-regularity of $(\nu_{p^l}(f_{\theta}(n)))_{n \geq 0}$. This in turn gives $p$-regularity of $(\nu_{p^l}(f(n)))_{n \geq 0}$ by the above discussion.  

At the same time, if this sequence were $k$-regular for some $k$ multiplicatively independent with $p$, then for any positive integer $u$ and any $\theta \in \R_f$ the sequence $(\nu_{p^l}(f_{\theta}(n)) \bmod{u})_{n \geq 0}$ would be simultaneously $p$- and $k$-automatic. Due to Cobham's Theorem it would then be eventually periodic, and this is not the case when $u \nmid m_f(\theta)$, as Proposition \ref{lem:eventually_periodic_valuation} shows.

In the case (c) Proposition \ref{prop:not_automatic_digits} says that $(\nu_{p^l}(f_{\theta}(n)))_{n \geq 0}$ is not regular for any $\theta \in \R_f$. Therefore, $(\nu_{p^l}(f(n)))_{n \geq 0}$ cannot be regular either.
\end{proof}

\begin{proof}[Proof of Theorem \ref{thm:prime_power_last_nonzero}]
In the case (b) the first factor on the right-hand side of \eqref{eq:last_nonzero_representation} is $p$-regular due to Proposition \ref{prop:rational_root}(ii). The second factor is $p$-regular as well as a polynomial in $n$, and therefore $(\LL_{p^l}(f_{\theta}(n)))_{n \geq 0}$ is $p$-regular as a termwise product of such sequences. Hence, $(\LL_{p^l}(f(n)))_{n \geq 0}$ is $p$-regular too. At the same time, it cannot be $k$-regular for $k$ multiplicatively independent with $p$, as then for all $d$ the sequence $(\ell_{p^l,d}(f(n)))_{n \geq 0}$ would be $k$-automatic, and this possibility is ruled out by Theorem \ref{thm:prime_power_d_last_nonzero}(b).

In a similar fashion, in the case (c) Theorem \ref{thm:prime_power_d_last_nonzero}(c) implies  that $(\LL_{p^l}(f(n)))_{n \geq 0}$ is not $k$-regular for any $k$.
\end{proof}

\begin{proof}[Proof of Theorem \ref{thm:prime_power_d_last_nonzero}]
If $\R'_f = \varnothing$, as in (a), then for all $\theta \in \R_f$ the condition \eqref{eq:condition} holds with  $c= c_{\theta}$ and $m = m_f(\theta)$. Here it is important that $l \mid T$ so that $\nu_p(c_{\theta}) \equiv  \nu_p(g_{\theta}(\theta)) \pmod{l}$. By Proposition \ref{prop:eventually_periodic_digits} applied to \eqref{eq:d_last_nonzero_representation}, this implies that $(\ell_{p^l,d}(f_{\theta}(n)))_{n \geq 0}$ is periodic for all $\theta \in \R_f$, and thus so is $(\ell_{p^l,d}(f(n)))_{n \geq 0}$.

In the case (b), the condition \eqref{eq:condition} holds for all irrational $\theta \in \R_f$ in the same sense as above. Propositions \ref{prop:eventually_periodic_digits} and \ref{prop:rational_root}(iii) together imply $p$-automaticity of $(\ell_{p^l,d}(f(n)))_{n \geq 0}$.  At the same time, there exists some rational $\theta \in \R_f$ for which \eqref{eq:condition} does not hold. As a consequence, $(\ell_{p^l,d}(f_{\theta}(n)))_{n \geq 0}$ is not eventually periodic, thus neither is $(\ell_{p^l,d}(f(n)))_{n \geq 0}$.

Finally, in the case (c) there exists some irrational $\theta \in \R_f$ for which \eqref{eq:condition} does not hold. By Proposition \ref{prop:not_automatic_digits}, the sequence $(\ell_{p^l,d}(f_{\theta}(n)))_{n \geq 0}$ is not automatic so $(\ell_{p^l,d}(f(n)))_{n \geq 0}$ is not automatic either.
\end{proof}

\section{Proofs for $b$ having several prime factors} \label{sec:several_factors_proofs}
Like before, let $b \geq 2$ be an integer base with prime factorization
$$	b = p_1^{l_1} \cdots p_s^{l_s},	$$
where $p_1,\ldots, p_s$ are distinct primes and $l_1,\ldots,l_s$ are positive integers. For $i=1,\ldots,s$ we put $b_i=p_i^{l_i}$. Throughout the whole section we assume that $s \geq 2$, unless specified otherwise. We retain the notation from the preceding sections, namely  $\Q_b, \Z_b, \A_b, \PP_b, \R_f$ and so on.

To begin, we prove Propositions \ref{prop:not_zero} and \ref{prop:regular_polynomial}, dealing with degenerate cases.

\begin{proof}[Proof of Proposition \ref{prop:not_zero}]
Without loss of generality we can assume $i=s$, so that $f_s = 0$, $\overline{b} = b/b_s$, and $\overline{f} = (f_1,\ldots,f_{s-1})$. Since $\nu_{b_s}(f_s(n)) = + \infty$  for all $n \in \N$, we get 
$$\nu_b(f(n)) = \nu_{\overline{b}}(\overline{f}(n))$$
which immediately gives (i).

By the above equality we can write 
$$ \LL_b(f(n)) = \left(\LL_{\overline{b}}(\overline{f}(n)), 0 \right).$$
Since the zero sequence is $k$-regular for all $k \geq 2$, by Proposition \ref{prop:function_regular}(iv) we obtain (ii).

Finally, let $(\gamma_1(n), \ldots, \gamma_s(n))$ be the image of $\ell_{b,d}(f(n))$ through the natural isomorphism between $\Z/b^d\Z$ and $\Z/b_1\Z \times \cdots \times \Z/b_s\Z$. We can deduce from Proposition \ref{prop:function_regular}(iv) that the sequence $(\ell_{b,d}(f(n)))_{n \geq 0}$ is $k$-automatic if and only if for each $i =1,\ldots,s$ the sequence $(\gamma_i(n))_{n \geq 0}$ is $k$-automatic. By our assumption $\gamma_s(n) = 0$ for all $n \in \N$, and so this sequence is $k$-automatic for all $k \geq 2$. Since the same equivalence holds for $\ell_{\overline{b},d}(\overline{f}(n))$ and $(\gamma_1(n),\ldots,\gamma_{s-1}(n))$, we obtain (iii).
\end{proof}

In the following proof we allow $s=1$.

\begin{proof}[Proof of Proposition \ref{prop:regular_polynomial}]
There exist $T, a \in \N$ such that for each $i=1,\ldots,s$ the function $h = (h_1,\ldots,h_s)$, given by $h(x) = f(b^T x + a)$, has no roots in $\Z_{b}$. (One can take $T$ sufficiently large and $a \not \equiv \theta \pmod{b^T}$.)  Proposition \ref{prop:several_factors_no_root}(i) below says that the sequence $(\nu_b(h(n)))_{n \geq 0}$ is periodic with period being a power of $b$. By increasing $T$ if necessary, we can assume that it is constant. Letting $v = \nu_b(h(n))$, we obtain 
$$ \LL_b(h(n)) = b^{-v}h(n).  $$
If the original sequence $(\LL_b(f(n)))_{n \geq 0}$ is $k$-regular, then Proposition \ref{prop:last_digits_factorization} implies that so is  $(b^{-v} h_i(n))_{n \geq 0}$ for $i=1,\ldots,s$. It is thus enough to prove for a prime $p$ and $h \in \A_p$ that if $(h(n))_{n \geq 0}$ is $k$-regular, then $h$ is a polynomial. 

Write
$$h(x) = \sum_{m=0}^{\infty} a_m x^m,$$
where $a_m \in \Q_p$ for all $m \in \N$. Consider the $\Z$-submodule generated by the family of subsequences $\mathbf{h}_j = (h(k^j n))_{n \geq 0}$, where $j \in \N$. If $(h(n))_{n \geq 0}$ is $k$-regular, then said $\Z$-submodule is generated by $\mathbf{h}_0, \mathbf{h}_1, \ldots, \mathbf{h}_J$  for some $J \in \N$. In particular, there exist $\beta_0, \beta_1, \ldots, \beta_J \in \Z$ such that for all $n \in \N$ we have
$$  h(k^{J+1} n) = \sum_{j=0}^{J} \beta_j h(k^j n).$$
As $\N$ is dense in $\Z_p$, the coefficients on both sides must be equal, which yields
$$ a_m \left(k^{(J+1)m} - \sum_{j=0}^{J} \beta_j k^{jm}\right) = 0$$
for all $m \in \N$. However, as $m$ tends to infinity, then so does the expression in the parentheses (as a real number) so only finitely many $a_m$ can be nonzero.
\end{proof}

From now on we will only deal with $f = (f_1,\ldots,f_s)$ having nonzero components, namely $f \in \A_b$. The general structure of our reasoning leading to Theorems \ref{thm:several_factors_valuation}, \ref{thm:several_factors_last_nonzero}, and \ref{thm:several_factors_d_last_nonzero} is similar as in the previous section. The main idea is again to reduce the general case to considering $f$ such that $f_i$ either has no root in $\Z_{p_i}$ or is of the form $c_i(x - \theta_i)^{m_i}$. However the technical details involved in the auxiliary results are rather different.

The following technical lemma plays a similar role to Lemma \ref{lem:different_val_digits}.

\begin{lem} \label{lem:valuation_change}
Let $\rho_i, \sigma_i \in \Z_{p_i}$ for $i=1,\ldots,s$ and assume that $(\rho_1 ,\ldots,\rho_s) \neq (\sigma_1,\ldots,\sigma_s)$. Let $m_i$ be positive integers and $c_i, d_i \in \Q_{p_i}$. Then we have the following.
\begin{enumerate}[label={\textup{(\roman*)}}]
\item There exists $j \in\{1,\ldots,s\}$ and infinitely many $n \in \N$ such that
\begin{align*}
\nu_{b_j}(c_j(n-\rho_j)^{m_j}) &> \nu_b\left(c_1(n-\rho_1)^{m_1},\ldots, c_s(n-\rho_s)^{m_s}\right),  \\
\nu_{b_j}(d_j(n-\sigma_j)^{m_j}) &= \nu_b\left(d_1(n-\sigma_1)^{m_1},\ldots, d_s(n-\sigma_s)^{m_s}\right). 
\end{align*}
\item Let $C \in \N$ be such that $C \geq 3 + 2\max_{1 \leq i \leq s} (m_i/l_i)$ and $\nu_{p_i}(c_i) \equiv \nu_{p_i}(d_i) \pmod{l_i C}$ for all $i=1,\ldots,s$. Then there exist infinitely many $n$ such that
$$ \nu_b\left(c_1(n-\rho_1)^{m_1},\ldots, c_s(n-\rho_s)^{m_s}\right) \not\equiv \nu_b\left(d_1(n-\sigma_1)^{m_1},\ldots, d_s(n-\sigma_s)^{m_s}\right) \pmod{C}. $$
\end{enumerate} 
\end{lem}
\begin{proof}
Up to renumbering, we can assume that $\rho_i \neq \sigma_i$ for $i \leq I$, for some $I \in \{2,\ldots,s\}$, and $\rho_i = \sigma_i$ otherwise. In particular, we have at least $\rho_1 \neq \sigma_1$. 

In part (i) we will show that the assertion holds with $j=1$. It is sufficient to find infinitely many $n$ satisfying the inequalities
\begin{align} 
\nu_{b_2}(c_2(n-\rho_2)^{m_2}) &< \nu_{b_1}(c_1(n-\rho_1)^{m_1}), \label{eq:valuation_change1} \\
\nu_{b_1}(d_1(n-\sigma_1)^{m_1}) & \leq \min_{2 \leq i \leq s} \nu_{b_i}(d_i(n-\sigma_i)^{m_i}). \label{eq:valuation_change2}
\end{align}
To begin, observe that for any choice of $v_1,\ldots,v_s \in \N$ there exist infinitely many $n \in \N$ such that 
$$\begin{cases}
\nu_{p_1}(n-\rho_1) = v_1, \\
\nu_{p_i}(n-\sigma_i) = v_i, &i \geq 2,
\end{cases}
$$
which is a consequence of the Chinese Remainder Theorem.

If $\rho_2 \neq \sigma_2$, then as we increase $v_1,\ldots,v_s$, the right-hand side  of both \eqref{eq:valuation_change1} and \eqref{eq:valuation_change2} grows to infinity, while the left-hand side remains bounded. Hence, it is enough to take $v_1, \ldots, v_s$ sufficiently large, and the corresponding values $n$ all satisfy the assertion.
If $\rho_2 = \sigma_2$, choosing large $v_1,\ldots,v_s$ also works, except in order for \eqref{eq:valuation_change1} to hold we additionally require $v_1,v_2$ to satisfy 
$$ \left\lfloor \frac{\nu_{p_2}(c_2)+m_2 v_2}{l_2} \right\rfloor  < \left\lfloor \frac{\nu_{p_1}(c_1) + m_1 v_1}{l_1} \right\rfloor.$$

We now move on to part (ii). Put $u_i =  \nu_{p_i}(\rho_i - \sigma_i)$ for  $i \leq I$ and assume (up to further renumbering) that
\begin{equation} \label{eq:valuation_change_ii_0}
\frac{\nu_{p_1}(d_1) + m_1 u_1}{l_1} = \min_{i \leq I} \frac{\nu_{p_i}(d_i) + m_i u_i}{l_i}.
\end{equation}
In order to prove the assertion it is sufficient to find infinitely many $n \in \N$ such that
\begin{align}
\nu_{b_1}(c_1(n-\rho_1)^{m_1}) &\leq \min_{2 \leq i \leq s} \nu_{b_i}(c_i(n-\rho_i)^{m_i}), \label{eq:valuation_change_ii_1}  \\
\nu_{b_1}(d_1(n-\sigma_1)^{m_1}) &\leq \min_{2 \leq i \leq s} \nu_{b_i}(d_i(n-\sigma_i)^{m_i}), \label{eq:valuation_change_ii_2}
\end{align}
and also
\begin{equation} \label{eq:valuation_change_ii_3}
\nu_{b_1}(c_1(n-\rho_1)^{m_1}) \not \equiv \nu_{b_1}(d_1(n-\sigma_1)^{m_1}) \pmod{C}.
\end{equation} 
This time, we put
$$v_1 = u_1 + 2\left\lceil  \frac{l_1}{m_1} \right\rceil $$ 
and choose $v_2,\ldots, v_s$ in such a way that for all $i=2,\ldots,s$ we have 
\begin{align}
u_i &< v_i, \nonumber \\
\frac{\nu_{p_1}(c_1) + m_1 v_1}{l_1} &\leq \frac{\nu_{p_i}(c_i) + m_i v_i}{l_i}, \label{eq:valuation_change_ii_4}
\end{align}
and also for all $i > I$ (if any)
\begin{equation} \label{eq:valuation_change_ii_5}
\frac{\nu_{p_1}(d_1) + m_1 u_1}{l_1} \leq \frac{\nu_{p_i}(d_i) + m_i v_i}{l_i}.
\end{equation}
Again, there exist infinitely many $n\in \N$ simultaneously satisfying the equalities
$$\nu_{p_1}(n-\rho_i) = v_i$$
for all $i=1,\ldots,s$. For such $n$ the inequality \eqref{eq:valuation_change_ii_1} is precisely the condition \eqref{eq:valuation_change_ii_4}. 
Moreover, the inequalities  $v_i > u_i$ for $i \leq I$  yield
$$ \nu_{p_i}(n- \sigma_i) = \begin{cases}
u_i &\text{if } i \leq I, \\
v_i &\text{if } i > I.
\end{cases}   $$
Combined with \eqref{eq:valuation_change_ii_0} (for $i \leq I$) and \eqref{eq:valuation_change_ii_5} (for $i > I$), this gives \eqref{eq:valuation_change_ii_2}. 
Finally, by the assumption $\nu_{p_1}(c_1) \equiv \nu_{p_1}(d_1) \pmod{l_1 C}$, we obtain
\begin{align*}
 \nu_{b_1}(c_1(n-\rho_1)^{m_1}) - \nu_{b_1}(d_1(n-\sigma_1)^{m_1}) &\equiv  \left\lfloor \frac{\nu_{p_1}(c_1) + v_1 m_1}{l_1} \right\rfloor   - \left\lfloor \frac{\nu_{p_1}(d_1) + u_1 m_1}{l_1} \right\rfloor  \\
 & \equiv \left\lfloor \frac{v_1 m_1}{l_1} \right\rfloor   - \left\lfloor \frac{u_1 m_1}{l_1} \right\rfloor \pmod{C}
\end{align*}
By the choice of $v_1$ and the inequality $ x- y -1  \leq \lfloor x \rfloor - \lfloor y \rfloor \leq x- y + 1$ valid for any real $x,y$, we get
$$ 0 < \left\lfloor \frac{v_1 m_1}{l_1} \right\rfloor   - \left\lfloor \frac{u_1 m_1}{l_1} \right\rfloor < C.$$
The non-congruence \eqref{eq:valuation_change_ii_3} follows.
\end{proof}

We now treat the case where for $f = (f_1,\ldots,f_i)$ the functions $f_i$ have no root in $\Z_{p_i}$. In contrast to Proposition \ref{prop:no_root}, here all the considered conditions are equivalent, as there is no counterpart to \eqref{eq:condition}. For the sake of convenience while proving necessity we allow each $f_i$ to have at most one root, an assumption which will ultimately be relaxed.

\begin{prop} \label{prop:several_factors_no_root}
Let $f \in \mathcal{A}_b$ be such that for each $i \in \{1,\ldots,s\}$ the function $f_i$ either has no root in $\Z_{p_i}$ or it is of the form
\begin{equation} \label{eq:several_factors_no_root_1}
 f_i(x) = c_i (x - \theta_i)^{m_i},
\end{equation}
for some $c_i \in \Q_{p_i}$, $\theta_i \in \Z_{p_i}$, and positive integer $m_i$.
Then the following conditions are equivalent:
\begin{enumerate}[label={\textup{(\roman*)}}]
\item for some $i \in \{1,\ldots,s\}$ the function $f_i$ has no root in $\Z_{p_i}$;
\item the sequence $(\nu_b(f(n)))_{n \geq 0}$ is eventually periodic;
\item for all $d \geq 1$ the sequence $(\ell_{b,d}(f(n)))_{n \geq 0}$ is eventually periodic;
\item the sequence $(\ell_{b}(f(n)))_{n \geq 0}$ is eventually periodic.
\end{enumerate}
If either case holds, then the sequences $(\nu_b(f(n)))_{n \geq 0}$ and $(\ell_{b,d}(f(n)))_{n \geq 0}$ are purely periodic and a~power of $b$ can be chosen as a~period.

Furthermore, if $f \in \PP_b$, then the conditions \textup{(i)--(iv)} are also equivalent to:
\begin{enumerate}
\item[\textup{(v)}] the sequence $(\LL_{b}(f(n)))_{n \geq 0}$ is $k$-regular for all $k \geq 2$.
\end{enumerate}
\end{prop}
\begin{proof}
We will prove the chain of implications (i)$\implies$(ii)$\implies$(iii)$\implies$ (iv) $\implies$(i). In the case when $f \in \PP_b$ we additionally show that (ii)$\implies$(v)$\implies$(iii).

First, if some $f_i$ has no root in $\Z_{p_i}$, there exists $V \in \N$ such that $\nu_{b_i}(f_i(n)) < V$ for all $n \in \N$, and therefore also $\nu_b(f(n)) < V$. By uniform continuity of $f_i$ there exists $T_i \in \N$ such that for all $x,y \in \Z_{p_i}$ we have 
$$
\nu_{b_i}(f_i(x+b_i^{T_i}y)- f_i(x)) \geq V.
$$
Letting $b^T$ be a common multiple of  $b_1^{T_1}, \ldots, b_s^{T_s}$ and putting $y = b^T / b^{T_i} $, we can write
\begin{equation} \label{eq:several_factors_no_root_2}
 f_i(x+b^T) = f_i(x) + b_i^{V} h_i(x),
\end{equation}
where $h_i(x) \in \Z_{p_i}$. For each $i =1,\ldots,s$ let 
$$  v_i(x) = \min\{ \nu_{b_i}(f_i(x)), V\}.$$
We have $\nu_b(f(n)) = \min_{1 \leq i \leq s} v_i(n)$ and, at the same time, \eqref{eq:several_factors_no_root_2} implies $v_i(n+b^T) = v_i(n)$ for all $n \in \N$. Therefore, we also get
$$\nu_b(f(n+b^T)) = \nu_b(f(n)),  $$
which proves (ii). In particular, the sequence $(\nu_b(f(n)))_{n \geq 0}$ is purely periodic and a power of $b$ is a period.

Now assume that (ii) holds and assume without loss of generality that the sequence $(\nu_b(f(n)))_{n \geq 0}$ is purely periodic (by shifting the original sequence). Let $\pi$ be its period. Then for each $a=0,1,\ldots,\pi-1$ and all $n \in \N$ we have
$$
\LL_b(f(\pi n + a)) = b^{-\nu_b(f(a))} f(\pi n + a).
$$
Fix $a$ and let $g_i(x)$ denote the $i$th component of $\LL_b(f(\pi x + a))$.
Then $g_i \in \A_{p_i}$ and uniform continuity shows that the sequence $(g_i(n) \bmod{b_i^d})_{n \geq 0}$ is purely periodic, where a power of $p_i$ can be chosen as a period. Since $\ell_{b,d}(f(\pi n + a)) = (g_1(n),\ldots,g_s(n)) \bmod{b^d}$, we obtain that $(\ell_{b,d}(f(\pi n + a)))_{n \geq 0}$ is also purely periodic, where a power of $b$ can be chosen as a period. A similar claim thus holds for the whole sequence $(\ell_{b,d}(f(n)))_{n \geq 0}$ which proves (iii). Moreover, observe that if $\pi = b^T$, then a power of $b$ is also a period of this sequence.

In the case when $f \in \PP_b$, each component of $\LL_b(f(\pi n + a))$ is a polynomial in $n$ so it constitutes a sequence $k$-regular for all $k \geq 2$. Hence, (v) follows from (ii). 
At the same time, (v) immediately implies (iii) due to Cobham's theorem.

The implication from (iii) to (iv) is trivial.

Finally, we prove that (iv) implies (i).  Suppose that $(\ell_{b}(f(n)))_{n \geq 0}$ is eventually periodic with period $\pi$ but all $f_i$ are of the form \eqref{eq:several_factors_no_root_1}. We are going to arrive at a contradiction with eventual periodicity of the characteristic sequences of the set
$$\{n \in \N: \ \nu_{b_j}(f_j(n)) = \nu_{b}(f(n))\}$$
for some $j \in \{1,\ldots,s\}$.
In order to do this we apply Lemma \ref{lem:valuation_change}(i) to $\rho_i = \theta_i, \sigma_i = \theta_i - \pi$, and $d_i = c_i$. As a result, for some $j \in \{1,\ldots,s\}$ we obtain infinitely many $n \in \N$ such that 
$$ \nu_{b_j}(f_j(n)) > \nu_b(f(n))   $$
but also
$$ \nu_{b_j}(f_j(n+\pi)) = \nu_b(f(n+\pi)),   $$
thus a contradiction.
\end{proof}

In the following few results we assume that 
$$f_i(x) = c_i(x - \theta_i)^{m_i},$$ 
where $c_i \in \Q_{p_i}$, $\theta_i \in \Z_{p_i}$ and $m_i$ is a positive integer.
We first study $k$-regularity of the considered sequences ich the case $\theta_1 = \ldots = \theta_s$, where $k$ is a specific value depending on $m_1, \ldots, m_s$.

\begin{prop} \label{prop:several_factors_rational_root}
Let $\theta \in \Q \cap \Z_b$. For each $i=1,\ldots,s$ let
$$f_i(x) = c_i (x-\theta)^{m_i}.$$
Let $k = b_1^{w_1} \cdots b_s^{w_s}$ where $w_1,\ldots,w_s$ are positive integers satisfying 
$$m_1w_1 = \cdots = m_s w_s.$$ 
We have the following:
\begin{enumerate}[label={\textup{(\roman*)}}]
\item the sequence $(\nu_b(f(n)))_{n \geq 0}$ is $k$-regular; 
\item if $m_1 = \cdots = m_s$, then the sequence $(\LL_b(f(n)))_{n \geq 0}$ is $b$-regular;
\item for all $d \geq 1$ the sequence $(\ell_{b,d}(f(n)))_{n \geq 0}$ is $k$-automatic.
\end{enumerate}
\end{prop}
\begin{proof}
Writing $\theta = q/r$ in lowest terms with $r > 0$, we obtain
$$f_i(n) = \frac{c_i}{r^{m_i}}(rn-q)^{m_i}$$
By Proposition \ref{prop:arithmetic_prog}, without loss of generality we may thus assume that $\theta = 0$, namely
$$ f_i(x) =c_i x^{m_i}.$$

We first prove parts (i) and (iii) simultaneously. For the sake of (iii), by raising $k$ to a~suitable power, we can assume that $\ell_{b_i,d}(k) = 1$ and $(b/b_i)^D \equiv 1 \pmod{b_i^d}$ for each $i=1,\ldots,s$, where $D$ denotes the common value of $m_iw_i$. 

Consider the subsequences $(f(kn+a))_{n \geq 0}$ with $a=0,1,\ldots,k-1$. For $a=0$ have
$$f(kn) = (k^{m_1},\ldots,k^{m_s}) \cdot f(n) = b^D K f(n),$$
where $K = b^{-D}(k^{m_1},\ldots,k^{m_s}) \in \Z_b$.
Since $\nu_{p_i}(b^{-D} k^{m_i}) = 0$, 
we obtain 
$$ \nu_b(f(kn)) = D + \nu_b(f(n)).  $$
By the conditions imposed on $k$, we get
$$ \ell_{b_i,d}(b^{-D} k^{m_i}) \equiv b^{-D} k^{m_i} \equiv  \left (\frac{b_i}{b}\right)^{D} \left (\frac{k}{b_i^{w_i}}\right)^{m_i} \equiv (\ell_{b_i,d}(k))^{m_i} \equiv 1 \pmod{b_i^d}. $$ 
This implies $\ell_{b,d}(K) = 1,$ and hence Proposition \ref{prop:multiplicative_properties}(ii) yields
$$ \ell_{b,d}(f(kn)) \equiv \ell_{b,d}(K) \ell_{b,d}(f(n)) \equiv \ell_{b,d}(f(n)) \pmod{b^d}. $$

Now, for $a \neq 0$ we must have $\nu_{p_i}(a) < \nu_{p_i}(k)$ for some $i$ so the function $f_i(kx+a)$ of $x \in \Z_{p_i}$ has no root in $\Z_{p_i}$. As a consequence of Proposition \ref{prop:several_factors_no_root}, both sequences $(\nu_b(f(kn+a)))_{n \geq 0}$ and $(\ell_{b,d}(f(kn+a)))_{n \geq 0}$ are periodic, and thus $k$-automatic. An argument similar to the one in Proposition \ref{prop:last_nonzero_digits_regular} gives the assertion of (i) and (iii).

In part (ii), if we let $m$ denote the common value of $m_i$, then $f(bn) = b^m f(n)$, and consequently 
$$\LL_b(f(bn)) = \LL_b(f(n)).$$
For each $a =1,\ldots,b-1$ the sequence $(\LL_b(f(bn+a)))_{n \geq 0}$  is $b$-regular by Proposition \ref{prop:several_factors_no_root}, and again we deduce $b$-regularity of $(\LL_b(f(n)))_{n \geq 0}$.
\end{proof}

On the other hand, in the case when $\theta_1, \ldots, \theta_s$ are not all equal, we obtain nonregular sequences.

\begin{prop} \label{prop:several_factors_irrational}
For each $i=1,\ldots,s$ let
$$f_i(x) = c_i (x-\theta_i)^{m_i}. $$
If $\theta_1,\ldots,\theta_s$ are not all equal, then:
\begin{enumerate}[label={\textup{(\roman*)}}]
\item the sequence $(\nu_b( f(n)))_{n \geq 0}$ is not regular;
\item the sequence $(\LL_b(f(n)))_{n \geq 0}$ is not regular;
\item the sequence $(\ell_{b,d}(f(n)))_{n \geq 0}$ is not automatic for any $d \geq 1$.
\end{enumerate}
\end{prop}
\begin{proof}
To begin, observe that (ii) follows immediately from (iii), as the reduction modulo $b^d$ of a $k$-regular sequence is $k$-regular. By the same argument, in (iii) it is sufficient to consider $d=1$.

Hence, we shall prove that the sequences $(\nu_b( f(n)))_{n \geq 0}$ and $(\ell_{b}(f(n)))_{n \geq 0}$ are not $k$-regular for any integer $k \geq 2$. We begin with some preparatory steps. Let $C \geq 3 + 2\max_{1 \leq i \leq s} (m_i/l_i)$ be an integer, as in the statement of Lemma \ref{lem:valuation_change}(ii).
By raising $k$ to a suitable power, we may assume that it is of the form
$$k = b_1^{w_1} \cdots b_s^{w_s}e,$$
where $w_1, \ldots, w_s, e$ are nonnegative integers such that $C \mid w_i$ for $i=1,\ldots,s$ and $b,e$ are coprime. 

Let $(n_t)_{t \geq 0}$ be the sequence integers given by
\begin{equation*}
\left\{
    \begin{aligned}
n_t &\equiv \theta_1 \pmod{b_1^{w_1t}}, \\
&\;\vdots \\
n_t &\equiv \theta_s \pmod{b_s^{w_st}},
 \end{aligned}
\right.
\end{equation*}
where $0 \leq n_t < (b_1^{w_1} \cdots b_s^{w_s})^t \leq k^t$. In particular, if $b,k$ are coprime, then $n_t = 0$ for all $t \in \N$. We are going to consider $\nu_b$ and $\ell_b$ evaluated at the subsequences $(f(k^tn+n_t))_{n \geq 0}$. The $i$-th component of $f(k^tn+n_t)$ can be written in the form
$$f_i(k^tn+n_t) = c_i k^{t m_i} \left(n - \frac{\theta_i - n_t}{k^t} \right)^{m_i} $$
for all $n \in \N$. Fix any distinct $t,r \in \N$ and put $\rho_i = (\theta_i - n_t)/k^t$ and $\sigma_i = (\theta_i - n_r)/k^r$. 

We argue that the assumption $(\rho_1, \ldots, \rho_s) \neq  (\sigma_1, \ldots, \sigma_s)$ of Lemma \ref{lem:valuation_change} is satisfied.
Observe that the equality $\rho_i = \sigma_i$ holds
if and only if
$$ \theta_i = \frac{k^tn_r-k^rn_t}{k^t-k^r},$$
and the expression on the right-hand side does not depend on $i$. Since $\theta_1, \ldots,\theta_s$ are not all equal, we must have $\rho_i \neq \sigma_i$ for some $i$.

We now focus on part (i). If the sequence $(\nu_b(f(n)))_{n \geq 0}$ were $k$-regular, then its reduction modulo $C$ would be $k$-automatic. Since $\nu_{p_i}(c_ik^{t m_i}) \equiv \nu_{p_i}(c_i k^{r m_i}) \pmod{C l_i}$, all the assumptions of Lemma \ref{lem:valuation_change}(ii) are satisfied (where $c_i$ is replaced with $c_ik^{t m_i}$ and $d_i = c_i k^{r m_i}$). Hence, there exists $n \in \N$ (in fact, infinitely many) such that
$$\nu_{b}(f(k^t n +n_t)) \not \equiv \nu_{b}(f(k^r n +n_r)) \pmod{C}.$$
But this means that the $k$-kernel of $(\nu_b(f(n)) \bmod{C})_{n \geq 0}$ contains infinitely many distinct subsequences, a contradiction.

Similarly, if the sequence $(\ell_b(f(n)))_{n \geq 0}$ were $k$-automatic, then so would be the characteristic sequence of the set 
$$ \{ n \in \N: \ \nu_{b_j}(f_j(n)) = \nu_b(f(n))  \}$$ 
for each fixed $j \in \{1,\ldots,s\}$.
But then again, Lemma \ref{lem:valuation_change}(i) provides some $j \in \{1,\ldots,s\}$ and infinitely many $n \in \N$ such that
$$ \nu_{b_j}(f_j(k^t n + n_t)) > \nu_b(f(k^t n + n_t))$$
and
$$ \nu_{b_j}(f_j(k^r n + n_r)) = \nu_b(f(k^r n + n_r)).$$
Therefore, the $k$-kernel of said characteristic sequence contains infinitely many distinct subsequences.
\end{proof}

The final step before the proofs of the main results is to show nonregularity of the sequence $(\LL_b(f(n)))_{n \geq 0}$ when all $f_i$ have the same, unique root in $\Z_b \cap \Q$ but its multiplicity varies with $i$. To this end we need a standard fact which essentially says that for any base $b$ the $b$-adic expansion of a rational number is eventually periodic. Since we have not been able to find a suitable reference for non-prime $b$, we provide a short proof for the sake of completeness. For $\theta \in \Z_b$ we extend the notation from the previous section by writing $\theta[t,b] = \theta \bmod{b^t}$ and $\theta\{t,b\} = (\theta - \theta[t,b])/b^t$ for each $t \in \N$. Here it is important to specify  the base $b$, as any $b'$ having identical prime factors to $b$ yields the same ring $\Z_b$ but (usually) distinct values $\theta[t,b'], \theta\{t,b'\}$. Only when there is no danger of confusion, $b$ will be suppressed from the notation. 

\begin{lem} \label{lem:periodic_approximation}
Let $b \geq 2$ be an integer and $\theta \in \Q \cap \Z_b$. Then the sequence $(\theta\{t,b\})_{t \geq 0}$ is eventually periodic.
\end{lem}
\begin{proof}
Considering $b$ fixed, we omit it in the notation. We first prove that $(\theta\{t+r\}) =(\theta\{t\})\{r\}$ for each fixed $t$ and all $r \in \N$. This is obvious for $r=0$. For $r=1$ we
first compute
$$ (\theta\{t\})[1] \equiv \frac{\theta-\theta[t]}{b^t} \equiv \frac{\theta[t+1]-\theta[t]}{b^t} \pmod{b},  $$ 
which implies
$$ (\theta\{t\})\{1\} =  \frac{\theta\{t\} - (\theta\{t\})[1]}{b} = \frac{\frac{\theta - \theta[t]}{b^t} - \frac{\theta[t+1] - \theta[t]}{b^t}}{b}  = \theta\{t+1\}.$$
For general $r$ the claim follows by induction.

Moving on to the statement of the lemma, write $\theta = v/u$ in lowest terms. Since $u \theta[t] \equiv v \pmod{b^t}$, the number $u \theta\{t\} = (v - u\theta[t])/b^t$ is an integer for each $t$.  Moreover, we have the bound
$$ \left| \frac{v - u\theta[t]}{b^t}\right| \leq |v| + |u| \frac{\theta[t]}{b^t} < |v| + |u|.$$
Therefore, there exist indices $r > t$ such that 
$$ \frac{\theta - \theta[r]}{b^{r}}  = \frac{\theta- \theta[t]}{b^{t}}, $$ 
which yields
$\theta\{r\} = \theta\{t\}$. Consequently, $\theta\{r+w\} = \theta\{t+w\}$ for all $w \in \N$.
\end{proof}

We can now prove the aforementioned result.

\begin{lem} \label{lem:several_factors_different_multiplicities}
Let $f = (f_1,\ldots,f_s) \in \PP_b$ and assume that for each $i=1,\ldots,s$ the function $f_i$ is of the form
$$f_i(x) = (x-\theta)^{m_i} g_i(x), $$
where $m_i \geq 1$ is an integer, $\theta \in \Q \cap \Z_b$, and $g_i \in \mathcal{P}_{p_i}$ is such that the values $\nu_{p_i}(g_i(x))$, $\ell_{p_i,l_i}(g_i(x))$ are constant with respect to $x \in \Z_{p_i}$.
If the exponents $m_i$ are not all equal, then the sequence $(\LL_b(f(n)))_{n \geq 0}$ is not regular. 
\end{lem} 
\begin{proof}
If the sequence $(\LL_b(f(n)))_{n \geq 0}$ were $k$-regular for some $k \geq 2$, then $(\ell_b(f(n)))_{n \geq 0}$ would be $k$-automatic.
Letting $c_i = p_i^{\nu_{p_i}(g_i(n))} \ell_{p_i,l_i}(g_i(n))$ for $i=1,\ldots,s$, we get
\begin{align*}
\ell_b(f(n)) &\equiv \ell_b\left(\frac{g_i(n)}{c_i}, \ldots, \frac{g_s(n)}{c_s}  \right) \ell_b\left(c_1 (n-\theta)^{m_1}, \ldots, c_s (n-\theta)^{m_s}  \right) \\
 &\equiv \ell_b\left(c_1 (n-\theta)^{m_1}, \ldots, c_s (n-\theta)^{m_s}  \right) \pmod{b}.
\end{align*} 
From Proposition \ref{prop:several_factors_rational_root} we deduce that $k = b_1^{w_1} \cdots b_s^{w_s},$ for some $w_1,\ldots,w_s \in \N$ satisfying $m_1w_1 = \ldots = m_s w_s$. 

By Lemma \ref{lem:periodic_approximation} the sequence  
$(\theta\{l,k\})_{l \geq 0}$ is eventually periodic with period $T$. More precisely, assume that $\theta\{l+T,k\} = \theta\{l+T,k\}$ for $l \geq L$. Now, $k$-regularity of $(\LL_b(f(n)))_{n \geq 0}$ implies $k$-regularity of $(\LL_b(f(k^{L} n  + \theta[L,k])))_{n \geq 0}$, and the components of $f(k^{L} n  + \theta[L,k])$ have the same form as those of $f(n)$, namely
$$
f_i(k^{L} n  + \theta[L,k])  =    (n - \theta\{L,k\})^{m_i} k^{Lm_i} g_i(k^{L} n  + \theta[L,k])).$$
Hence, without loss of generality we can assume that $L=0$. Moreover, we have $\theta\{l,k^T\} = \theta\{lT,k\}$ for all $l \in \N$, and thus, replacing $k$ with $k^T$, we can further assume that $\theta\{l,k\} = \theta$ is constant with respect to $l$. In the sequel we suppress $k$ in the notation and simply write $\theta[l] = \theta[l,k]$.

Renumber the primes so that $m_1$ is minimal among $m_1,\ldots,m_s$.
It is sufficient to show that the first component of $(\LL_b(f(n)))_{n \geq 0}$, namely $(b^{-\nu_b(f(n))} f_1(n))_{n \geq 0}$,
is not a $k$-regular sequence. More precisely, for all $l \in \N$ we define
$$  \beta_l(n) = b^{-\nu_b(f(k^ln + \theta[l]))} f_1(k^ln + \theta[l]).$$
and claim that the $\Z$-submodule generated by the family $\{(\beta_l(n))_{n \geq 0}: l \in \N\}$ is not finitely generated.   

First, we compute the exponent of $b$ in the formula defining $\beta_l(n)$. Writing
$$ f_i(k^{l} n  + \theta[l]) = k^{lm_i} (n - \theta)^{m_i} g_i(k^{l} n  + \theta[l])$$
and letting $D$ denote the common value of $m_i w_i$, we obtain
$$
\nu_{b_i}(f_i(k^{l} n  + \theta[l])) =  lD  + \nu_{b_i}(f_i(n)),  
$$
and thus also
\begin{equation} \label{eq:not_regular_val}
 \nu_{b}(f(k^{l} n  + \theta[l])) =  lD  + \nu_{b}(f(n)).  
\end{equation}

Now, for the sake of contradiction suppose that for some $t \in \N$ the sequences $(\beta_l(n))_{n \geq 0}$ with $l=0,1,\ldots,t$ generate said $\Z$-submodule. In particular there exist integers $\alpha_0,\alpha_1, \ldots, \alpha_t$ such that for all $n \in \N$ we have
$$
\sum_{l=0}^{t} \alpha_l \beta_l(n) = \beta_{t+1}(n).
$$
Using \eqref{eq:not_regular_val}, after some simplification we obtain
$$
\sum_{l=0}^{t} \alpha_l C^l  g_1(k^ln + \theta[l]) = C^{t+1} g_1(k^{t+1}n + \theta[t+1]),
$$
where $C =k^{m_1} / b^D$. Because $\N$ is dense in $\Z_{p_1}$, we can replace $n$ with $x \in \Z_{p_1}$. In particular, for $x = \theta$ we obtain $g_1(k^l\theta + \theta[l]) = g_1(k^{t+1}\theta + \theta[t+1]) = g_1(\theta)$,
and thus 
\begin{equation} \label{eq:compare_const}
\sum_{l=0}^{t} \alpha_l C^l  = C^{t+1}.
\end{equation}
However, since not all $m_i$ are equal and $m_1$ is minimal among them, $C$ is a rational number lying in the interval $(0,1)$. After reducing both sides of \eqref{eq:compare_const} to the lowest terms, the denominator on the right-hand side remains larger, thus a contradiction.
\end{proof}

We now move on to prove Theorems \ref{thm:several_factors_valuation}, \ref{thm:several_factors_last_nonzero}, and \ref{thm:several_factors_d_last_nonzero}. As in the prime power case, we first make some preparations. Consider $f = (f_1,\ldots,f_s) \in \A_b$. If some $f_i$ has no root in $\Z_{p_i}$, then Proposition \ref{prop:several_factors_no_root} immediately implies part (a) of each theorem. Hence, in the following considerations we assume that  all the functions $f_i$ have a root in $\Z_{p_i}$, or equivalently, the set $\R_f$ is nonempty. For each $i=1,\ldots,s$ let $T_i \in \N$ be an integer obtained from Proposition \ref{prop:constant_val_digits} applied to $f_i$. Choose $T \in \N$ such that $T  \geq T_i/l_i$ for all $i=1,\ldots,s$, so that $b_i^T$ is a multiple of $p_i^{T_i}$. We will focus on the subsequences 
$(f(b^Tn+a))_{n \geq 0}$ with $a=0,1,\ldots,b^T-1$. 
If there is no $\theta = (\theta_1,\ldots,\theta_s) \in \R_f$ such that $a \equiv \theta \pmod{b^T}$, then at least one of the functions $f_i(b^Tx+a)$ of $x \in \Z_{p_i}$ has no root in $\Z_{p_i}$. For the same reason as in the prime power case, due to Proposition \ref{prop:several_factors_no_root} these subsequences $(f(b^Tn+a))_{n \geq 0}$ have no effect on $k$-regularity of $(f(n))_{n \geq 0}$. Therefore, it remains to consider $a = \theta[T,b]$ for $\theta \in \R_f$. To this end, we define the functions $f^{(\theta)} = (f^{(\theta)}_1, \ldots, f^{(\theta)}_s) \in \A_b$ (not to be confused with derivatives) by
$$f^{(\theta)}(x) = f(b^Tx + \theta[b,T]).$$
Its components may be written in the form
\begin{equation} \label{eq:f_theta_i}
f^{(\theta)}_{i}(x) = \left(x-\theta\{b,T\}_i\right)^{m_i} q^{(\theta)}_{i}(x), 
\end{equation}
where $\theta\{b,T\}_i = (\theta_i-\theta[b,T])/b^T$ is the only root of  $f^{(\theta)}_{i}$ in $\Z_{p_i}$, $m_i = m_{f_i}(\theta_i)$ is its multiplicity, and $q^{(\theta)}_{i} \in \A_{p_i}$. Moreover, the choice of $T$ also guarantees that for each $i=1,\ldots,s$ the values $\nu_{p_i}(q^{(\theta)}_{i}(x))$ and $\ell_{p_i,l_i d}(q^{(\theta)}_{i}(x))$ are constant with respect to $x \in \Z_{p_i}$.

Also, for $i=1,\ldots,s$ put
$$c^{(\theta)}_{i} = p_i^{\nu_{p_i}(q^{(\theta)}_{i}(x_1))} \ell_{p_i,l_i d}(q^{(\theta)}_{i}(x_s))$$
and for $x = (x_1,\ldots,x_s) \in \Z_b$ define
\begin{align*}
q^{(\theta)}(x) &= \left(q^{(\theta)}_{1}(x), \ldots, q^{(\theta)}_{2}(x)\right), \\
h^{(\theta)}(x) &= \left(c^{(\theta)}_{1} \left(x_1-\theta\{b,T\}_1\right)^{m_1}, \ldots, c^{(\theta)}_{s} \left(x_s-\theta\{b,T\}_{s}\right)^{m_s} \right).
\end{align*}
We can thus use Proposition \ref{prop:multiplicative_properties} to obtain for all $n \in \N$ the equalities
\begin{align*}
\nu_b(f^{(\theta)}(n)) &= \nu_b( h^{(\theta)}(n)), \\
\LL_b(f^{(\theta)}(n)) &= \LL_b( h^{(\theta)}(n)) \cdot  \frac{q^{(\theta)}(n)}{(c_1^{(\theta)},\ldots,c_s^{(\theta)})} , \\
\ell_{b,d}(f^{(\theta)}(n)) &= \ell_{b,d}( h^{(\theta)}(n)).
\end{align*}
The results proved earlier in this section can be applied to $h^{(\theta)}$ so by extension also to $f^{(\theta)}$.

\begin{proof}[Proof of Theorem \ref{thm:several_factors_valuation}]
In part (b) there exists precisely one $\theta = (\theta, \ldots, \theta) \in \R_f$, a rational number. Let $k = b^{w_1} \cdots b^{w_s}$, where $ w_1,\ldots,w_s$ are nonnegative integers satisfying $m_1 w_1 = \cdots = m_sw_s$. Proposition \ref{prop:several_factors_rational_root}(i) implies $k$-regularity of the sequence $(\nu_b(h^{(\theta)}(n)))_{n \geq 0}$, and therefore also of $(\nu_b(f(n))_{n \geq 0}$, by the above discussion. 

If this sequence were not strictly $k$-regular, then its reduction modulo any positive integer would be eventually periodic. Let $C$ be an integer as in Lemma \ref{lem:different_val_digits}(ii) and suppose that $\pi$ is a period of the subsequence $(\nu_b(h^{(\theta)}(n)) \bmod{C})_{n \geq 0}$. However, the lemma applied to $\rho_i = \theta\{b,T\}$ and $\sigma_i = \theta\{b,T\}- \pi$ for all $i=1,\ldots,s$ leads to a contradiction with eventual periodicity.

Finally, under the assumption of part (c) we can find irrational $\theta = (\theta_1, \ldots, \theta_s) \in \R_f$, namely such that not all $\theta_i$ are equal. For any such $\theta$ the sequences $(\nu_b(h^{(\theta)}(n)))_{n \geq 0}$ is not regular by Proposition \ref{prop:several_factors_irrational}(i), and the result follows.\end{proof}

\begin{proof}[Proof of Theorem \ref{thm:several_factors_last_nonzero}]
In part (b) we have $m_1 = \cdots = m_s$, and $b$-regularity of the sequence $(\LL_b(h^{(\theta)}(n)))_{n \geq 0}$ follows from Proposition \ref{prop:several_factors_rational_root}(ii).
Since each component $(q_i^{(\theta)}(n)/c^{(\theta)}_i)_{n \geq 0}$ is $k$-regular for all $k \geq 2$, the same holds for $(q^{(\theta)}(n)/(c^{(\theta)}_1,\ldots,c^{(\theta)}_s))_{n \geq 0}$. Thus, $(\LL_b(f^{(\theta)}(n)))_{n \geq 0}$ and consequently $(\LL_b(f^{(\theta)}(n)))_{n \geq 0}$ are $b$-regular.

If this sequence were not strictly $b$-regular, then the sequence $(\ell_b(f(n)))_{n \geq 0}$, and thus also $(\ell_{b,d}( h^{(\theta)}(n))_{n \geq 0}$ would be eventually periodic, but this is ruled out by Proposition \ref{prop:several_factors_no_root}.

In part (c) the set $\R_f$ contains some irrational $\theta$ or rational $\theta$  with varying multilplicities $m_{f_i}(\theta)$. In the former case nonregularity of the sequence $(\LL_b(h^{(\theta)}(n)))_{n \geq 0}$ follows from Proposition \ref{prop:several_factors_irrational}(ii). In the latter case we apply Lemma \ref{lem:several_factors_different_multiplicities} to the function $f^{(\theta)}$ with components written in the form \eqref{eq:f_theta_i}.
\end{proof}

\begin{proof}[Proof of Theorem \ref{thm:several_factors_d_last_nonzero}]
In part (b), $k$-automaticity of $(\ell_{b,d}(f(n)))_{n \geq 0}$  again follows from Proposition \ref{prop:several_factors_rational_root}(iii), while Proposition \ref{prop:several_factors_no_root} implies that this property holds in the strict sense.

Part (c) is a consequence of Proposition \ref{prop:several_factors_irrational}(iii).
\end{proof}

\section*{Acknowledgements}
The research is supported by the grant of the National Science Centre (NCN), Poland, no.\ UMO-2019/34/E/ST1/00094.

\bibliographystyle{amsplain}
\bibliography{references}
\end{document}